\documentclass[11pt,reqno]{amsart}
\IfFileExists{libertine.sty}{\usepackage[tt=false]{libertine}}{\usepackage{lmodern}}
\usepackage{amssymb}
\usepackage{mathtools}
\IfFileExists{newpxmath.sty}{\usepackage[varbb]{newpxmath}}{}
\usepackage[foot]{amsaddr}

\IfFileExists{mathabx.sty}{%
  \let\savedbigtimes\bigtimes
  \let\bigtimes\relax
  \usepackage{mathabx}
  \let\bigtimes\savedbigtimes
}{}

\usepackage[margin=1in, bottom=1in]{geometry}
\usepackage{graphicx}
\usepackage{caption}
\usepackage{subcaption}
\usepackage{enumitem}
\usepackage{comment}
\IfFileExists{bbm.sty}{\usepackage{bbm}}{}

\usepackage[dvipsnames]{xcolor}

\usepackage[colorlinks=true,
  linkcolor=teal!60!black,
  citecolor=PineGreen,
  urlcolor=RedViolet]{hyperref}
\hypersetup{bookmarksopen=true,pdftitle={Almost Sharp Equivalence between Approximate Message Passing and Low-Degree Polynomials},pdfauthor={Zhangsong Li}}

\usepackage[yyyymmdd,hhmmss]{datetime}

\usepackage{tikz}
\usetikzlibrary{decorations.pathreplacing,calligraphy}
\usetikzlibrary{calc}

\newtheorem{thm}{Theorem}[section]
\newtheorem{lem}[thm]{Lemma}

\theoremstyle{definition}

\theoremstyle{remark}
\newtheorem{rmk}[thm]{Remark}

\numberwithin{equation}{section}
\def\beq#1\eeq{%
    \begin{equation}%
    #1%
    \end{equation}%
}
\def\baln#1\ealn{%
    \begin{align}%
    #1%
    \end{align}%
}
\def\balnn#1\ealnn{%
    \begin{align*}%
    #1%
    \end{align*}%
}

\def\bC{{\boldsymbol{C}}}

\def\bG{{\boldsymbol{G}}}
\def\bH{{\boldsymbol{H}}}

\def\bR{{\boldsymbol{R}}}

\def\bW{{\boldsymbol{W}}}

\def\bY{{\boldsymbol{Y}}}
\def\bZ{{\boldsymbol{Z}}}

\def\btheta{{\boldsymbol{\theta}}}

\def\bv{{\boldsymbol{v}}}

\def\bbE{{\mathbb{E}}}

\def\bbN{{\mathbb{N}}}

\def\bbR{{\mathbb{R}}}

\def\cB{{\mathcal{B}}}
\def\cC{{\mathcal{C}}}

\def\cE{{\mathcal{E}}}

\def\cG{{\mathcal{G}}}

\def\cI{{\mathcal{I}}}
\def\cJ{{\mathcal{J}}}
\def\cK{{\mathcal{K}}}

\def\cN{{\mathcal{N}}}

\def\cS{{\mathcal{S}}}
\def\cT{{\mathcal{T}}}

\def\ber{{\mathsf{Ber}}}

\def\amp{{\mathsf{AMP}}}
\def\bayes{{\mathsf{Bayes}}}
\def\LD{{\mathsf{LD}}}
\def\MMSE{{\mathsf{MMSE}}}
\def\cum{{\mathsf{Cum}}}
\def\core{{\mathsf{Core}}}
\def\aut{{\mathsf{Aut}}}

\title[AMP and Low-Deg Equivalence]{Almost Sharp Equivalence between Approximate Message Passing and Low-Degree Polynomials}

\author{Zhangsong Li}

\address[Z.~Li]{School of Mathematical Sciences, Peking University}
\email[Z.~Li]{ramblerlzs@pku.edu.cn}

\date{\today}

\begin{document}

\begin{abstract}
    We prove a sharp lower bound for growing-degree polynomial estimation in the symmetric rank-one matrix estimation problem. The observation is
    \[
        \boldsymbol{Y}=\frac{1}{\sqrt{n}} \boldsymbol{\theta} \boldsymbol{\theta}^{\top}+\boldsymbol{W},
    \]
    where the coordinates of $\boldsymbol{\theta}$ are independent with a common, fixed sub-Gaussian distribution $\pi_\Theta$ satisfying $\bbE[\Theta]\neq0$, and $\boldsymbol{W}$ is symmetric with independent standard Gaussian upper-triangular entries. For every such prior, we give an explicit finite-dimensional bound implying that every sequence of polynomial estimators of degree $D(n)=o(n^{1/60})$ has normalized mean-square error with limit inferior at least $\bbE[\Theta^2]-q_\amp$, the limiting error of Bayes approximate message passing (AMP). This extends the constant-degree equivalence of Montanari and Wein~\cite{montanari2025equivalence} to growing degree under the same prior assumptions. Combined with their polynomial approximation of fixed-iteration AMP, the bound identifies the exact limiting low-degree MMSE whenever $D(n)\to\infty$ within this range. It thereby answers the growing-degree AMP-equivalence question discussed in~\cite{wein2025computational, maleki2026high} for nonzero-mean sub-Gaussian priors in the stated degree range. We also obtain a moment-dependent finite-dimensional lower bound for any nonzero-mean prior with all moments finite.

    The proof bounds low-degree MMSE using \emph{conditional} joint cumulants and a refined application of Jensen's inequality. The conditional cumulants are defined using an auxiliary Gaussian channel $\boldsymbol{R}$ calibrated to the AMP fixed point. This retains signal dependence that is lost in unconditional cumulant bounds and produces the cancellations needed for quantitative control as the degree grows. 
\end{abstract}

\maketitle

\setcounter{tocdepth}{1}
\tableofcontents

\section{Introduction and main results}
\label{sec:intro}

Statistical-computational gaps are a ubiquitous phenomenon in high-dimensional statistics. Consider estimation of a high-dimensional parameter vector $\btheta= (\theta_1,\theta_2,\dots,\theta_n)$ from noisy observations $\bY\sim {\rm P}_{\btheta}$. In many models, we can characterize the optimal accuracy achieved by arbitrary estimators. However, when we analyze classes of estimators that can be implemented via polynomial-time algorithms, significantly lower accuracy is obtained. A short list of examples includes sparse regression~\cite{celentano2022fundamental, gamarnik2022sparse}, sparse principal component analysis~\cite{johnstone2009consistency, amini2008high, krauthgamer2015semidefinite, berthet2013complexity}, graph clustering and community detection~\cite{decelle2011asymptotic}, the planted clique problem~\cite{barak2019nearly}, and tensor principal component analysis~\cite{montanari2014statistical, hopkins2015tensor}.

This state of affairs has led to the conjecture that these gaps are fundamental in nature: there simply exists no polynomial-time algorithm achieving the statistically optimal accuracy. Proving such a conjecture is extremely challenging since  standard complexity-theoretic assumptions (e.g.\ P$\neq$NP) are ill-suited to establish complexity lower bounds in \emph{average-case} settings where the input to the algorithm is random. Thus, currently the leading approaches for \emph{demonstrating} hardness are based on either average-case reductions which formally relate different average-case problems to each other (see, e.g., \cite{brennan2018reducibility, brennan2020reducibility} and references therein) or based on unconditional lower bounds against restricted classes of algorithms (see, e.g., \cite{diakonikolas2017statistical, gamarnik2021overlap}). However, the proliferation of lower bounds suggests a need for unifying principles, especially because the story is repeated for numerous statistical estimation problems: lower bounds against a variety of restricted computational models are proven independently, all usually pointing to the same signal-to-noise ratios tolerated by efficient algorithms. Thus, a problem of growing interest is whether some or all of these restricted computational models are equivalent in power. This paper focuses on establishing that approximate message passing (AMP) algorithms and low-degree (Low-Deg) polynomial estimators achieve asymptotically the same accuracy in the symmetric rank-one estimation problem.  

The motivation for studying the relation between AMP and Low-Deg comes from the distinctive position occupied by these two classes in the theoretical  landscape. AMP is a class of iterative algorithms motivated by ideas in information theory (iterative decoding~\cite{gallager1962low}) and statistical physics (cavity method and TAP equations~\cite{thouless1977solution, mezard1987spin}). Their structure is relatively constrained: they operate by alternating a matrix-vector multiplication (with a matrix constructed from the data $\bY$), and a non-linear operation on the same vectors (typically independent of the data matrix). This structure has opened the way to a sharp characterization of their behavior in the high-dimensional limit, known as ``state evolution'' \cite{donoho2009message, bayati2011dynamics, bolthausen2014iterative}.
On the other hand, Low-Deg was originally motivated by a connection to Sum-of-Squares (SoS) semidefinite programming relaxations and captures a different (algebraic) notion of complexity \cite{hopkins2017efficient, hopkins2017power, schramm2022computational, ding2025low}. In Low-Deg procedures, each unknown parameter $\btheta_i$ is estimated via a fixed polynomial in the entries of the data matrix $\bY$, of constant or moderately growing degree. As such, these estimators are relatively unconstrained, and indeed capture  (via polynomial approximation) a broad variety of natural approaches. Developing a sharp analysis of such a broad class of estimators can be very challenging.

In the symmetric rank-one estimation problem, we observe a noisy version $\bY$ of the rank-one matrix $\btheta\btheta^{\top}$, with $\btheta\in\bbR^n$ an unknown vector:
\begin{equation}{\label{eq:model-def}}
    \bY = \frac{1}{\sqrt{n}} \btheta\btheta^{\top} + \bW.
\end{equation}
Here $\bW$ is independent of $\btheta$, $\bW=\bW^{\top}$, and $(\bW_{ij})_{1\le i\le j\le n}$ are independent $\cN(0,1)$ variables, including the diagonal. We use this unit-diagonal-variance convention throughout the proof. The lower bound also holds for the conventional GOE normalization with diagonal variance two, as in Remark~\ref{rmk:diagonal-var-two}. Given a single observation of $\bY$, we would like to estimate $\btheta$.

Because of its simplicity, the rank-one estimation problem has been a useful sandbox to develop new mathematical ideas in high-dimensional statistics. A significant amount of work has been devoted to the analysis of spectral estimators, which typically take the form $\Hat{\btheta}(\bY) =c_n\, \bv_1(\bY)$ where $\bv_1(\bY)$ is the top eigenvector of $\bY$ and $c_n$ is a scaling factor~\cite{hoyle2004principal, baik2005phase, baik2006eigenvalues, benaych2011eigenvalues}. However, spectral methods are known to be suboptimal if additional information is available about $\btheta$ \cite{lesieur2015phase, banks2018information, perry2018optimality, el2020fundamental}. In this paper, we model this information by assuming $(\btheta_i)_{1\le i\le n} \overset{\mathrm{i.i.d.}}{\sim} \pi_{\Theta}$, for $\pi_{\Theta}$ a probability distribution on $\bbR$ (which does not depend on $n$). Throughout, we assume $\bbE[\Theta]\neq0$ and $\bbE[|\Theta|^p]<\infty$ for every $p\geq1$, where $\Theta\sim\pi_\Theta$, and write $m_2=\bbE[\Theta^2]$ and $\|\Theta\|_p=(\bbE[|\Theta|^p])^{1/p}$. Our main result assumes in addition that the prior is sub-Gaussian, which we express in the moment form
\[
    K=\max\left\{1,\sup_{p\geq2}\frac{\|\Theta\|_p}{\sqrt p}\right\}<\infty.
\]
In particular, the result includes the Bernoulli prior $\pi_\Theta=\ber(\rho)$, $0<\rho<1$, for which $m_2=\rho$ and $K=1$; the resulting model is widely studied under the name of the planted submatrix problem \cite{ma2015computational, schramm2022computational, sohn2025sharp}. The results below concern the Bayes risk under the i.i.d.\ prior.

The high-dimensional Bayes performance of model \eqref{eq:model-def} has been studied through scalar variational formulas; see, e.g., \cite{lelarge2017fundamental, barbier2019adaptive, chen2022hamilton}. We use the scalar potential below only to specify the AMP fixed point, and do not need a Bayes-MMSE limit at parameters with nonunique minimizers. The asymptotic accuracy of the optimal AMP algorithm (called Bayes AMP) is also known~\cite{montanari2021estimation}. It is useful to pause and recall these results in some detail. For an independent scalar pair $\Theta\sim\pi_\Theta$ and $Z\sim\cN(0,1)$, put (below $I(X;Y)$ is the mutual information between the variables $X$ and $Y$)
\begin{equation}{\label{eq:mutual-info}}
    I_{\pi_\Theta}(s)=I(\Theta;\sqrt s\Theta+Z), \qquad
    \Psi(q;b)=\frac{q^2}{4}-\frac{(m_2+b)q}{2}+I_{\pi_\Theta}(q), \qquad \Psi(q)=\Psi(q;0).
\end{equation}
Here $s,q\geq0$ and $b\in\bbR$.
In addition, define
\begin{align}{\label{eq:def-q-bayes-amp}}
    q_{\bayes}:=\min\operatorname*{arg\,min}_{q\geq0}\Psi(q); \qquad
    q_{\amp}:=\inf\{q\geq0:\Psi'(q)=0,\ \Psi''(q)\geq0\}.
\end{align}
The existence and positivity of $q_\amp$, and the scalar identities used in the proof, are verified in Lemma~\ref{lem:fixed-identity}. We recall the following AMP and Bayes-risk characterizations from \cite{montanari2021estimation, montanari2024statistically, montanari2025equivalence}. 
\begin{thm}{\label{thm:bayes-amp}}
    Assume $\pi_\Theta$ is independent of $n$, has nonzero mean, and has finite moments of all orders. There is a sequence of AMP estimators $\btheta^{\amp,t}(\bY)$, using $t$ iterations and the standard uninformative initialization corresponding to $q_0=0$ in scalar state evolution, such that
    \begin{align}{\label{eq:bayes-limit}}
        \lim_{t\to\infty}\lim_{n\to\infty} \frac1n \bbE\left[\|\btheta^{\amp,t}(\bY)-\btheta\|_2^2\right]
        =m_2-q_\amp.
    \end{align}
    For every $\varepsilon>0$, a fixed number $t=t(\pi_\Theta,\varepsilon)$ of iterations gives limiting MSE at most $m_2-q_\amp+\varepsilon$ in time $O_{\pi_\Theta,\varepsilon}(n^2)$.

    Further, assume that $b \mapsto \Psi_*(b):=\inf_{q \ge 0} \Psi(q;b)$ is differentiable at $b=0$. Then the minimum MSE of any estimator is
    \begin{align}{\label{eq:stats-limit}}
        \lim_{n\to\infty} \inf_{\widehat{\btheta}} \frac1n \bbE\left[\| \widehat{\btheta}(\bY)-\btheta \|_2^2 \right]
        =m_2-q_\bayes.
    \end{align}
\end{thm}
Here Bayes AMP uses scalar posterior-mean denoisers. When these are not globally Lipschitz, the estimators in the theorem use fixed Lipschitz approximations, with asymptotic risk error at most $1/t$ for the $t$-iteration procedure; the approximating functions may depend on $t,\pi_\Theta$, but not on $n$. This is the approximation formulation of AMP achievability in~\cite[Theorem~2.2]{montanari2025equivalence}. Changing the diagonal noise variance from two to one does not change fixed-iteration state evolution; one may equivalently omit the diagonal entries in these AMP iterations.

We now turn to the class of low-degree estimators. Define the low-degree minimum mean-square error (MMSE) by
\begin{equation}{\label{eq:low-deg-mmse}}
    \MMSE_{\le D}:=\inf_{\widehat\btheta\in\LD(n,D)} \frac1n\bbE \left[ \|\widehat\btheta(\bY)-\btheta\|^2 \right].
\end{equation}
Here $\LD(n,D)$ is the class of maps whose $n$ coordinates are real polynomials of total degree at most $D$ in the upper-triangular observations. Their coefficients may depend on $n,D,\pi_\Theta$ but not on the realized observations or latent variables. 
For a sub-Gaussian prior, Montanari and Wein~\cite[Theorem~2.3]{montanari2025equivalence} show that fixed-iteration Bayes AMP can be approximated by fixed-degree polynomial estimators. Applying their result to $\Theta$ as above, for every $\varepsilon>0$ there is a degree $D_\varepsilon$, independent of $n$, such that
\[
    \limsup_{n\to\infty}\MMSE_{\le D_\varepsilon} \le m_2-q_\amp+\varepsilon.
\]
The same work proves $\liminf_{n\to\infty}\MMSE_{\le D}\ge m_2-q_\amp$ for every fixed $D$, requiring only finite moments of all orders for this lower bound. These statements explain the equivalence between the two methods when the dimension tends to infinity before the degree is allowed to grow. They do not, by themselves, control a sequence $D=D(n)$: a remainder that vanishes for each fixed $D$ may become substantial when $D$ increases with $n$. Since degree-$D$ polynomials can be evaluated using $n^{O(D)}$ time, it is natural to ask how far this equivalence extends as $D$ grows with $n$. Extending this exact AMP-risk lower bound to growing degree is the question discussed in \cite{wein2025computational, maleki2026high} that we address here for arbitrary fixed nonzero-mean sub-Gaussian priors. Our main result can be summarized as follows.
\begin{thm}\label{thm:main}
The following holds with the universal constant $C_0=60$ (or any larger constant). Let $\pi_\Theta$ be a fixed sub-Gaussian prior with $\bbE[\Theta]\neq0$, and let $K$ be its moment parameter defined above. For every integer $n\geq1$ and integer $D\geq1$,
\begin{equation}\label{eq:main-finite}
    \MMSE_{\le D}\geq m_2-q_\amp -\sum_{g=1}^{D}\left(\frac{[2C_0 K(D+1)]^{C_0}}{n}\right)^g.
\end{equation}
In particular, if $[2C_0K(D+1)]^{C_0}\leq n/2$, the error term is at most $2[2C_0K(D+1)]^{C_0}/n$. Thus, for every fixed $\pi_\Theta$ and every $D(n)=o(n^{1/60})$, 
\[
    \liminf_{n\to\infty}\MMSE_{\le D(n)}\geq m_2-q_\amp.
\]
\end{thm}
Together with the fixed-degree approximation upper bound above and monotonicity in $D$, Theorem~\ref{thm:main} implies
\[
    \lim_{n\to\infty}\MMSE_{\le D(n)} =m_2-q_\amp \qquad\text{whenever }D(n)\to\infty\text{ and }D(n)=o(n^{1/60}).
\]
Our proof thus significantly strengthens \cite{montanari2025equivalence} in this setting. 
\begin{rmk}[Diagonal variance two]{\label{rmk:diagonal-var-two}}
    Theorem~\ref{thm:main} also applies when the diagonal noise has variance two. Indeed, let $\bY^{(1)}$ denote the unit-diagonal-variance observation and let $\bG$ be an independent diagonal matrix with i.i.d. $\cN(0,1)$ diagonal entries. Then $\bY^{(2)}=\bY^{(1)}+\bG$ has the diagonal-variance-two law. For any degree-$D$ polynomial estimator $\widehat{\btheta}$ in $\bY^{(2)}$, the map $\overline{\btheta}(y)=\bbE_{\bG}[\widehat{\btheta}(y+\bG)]$ is a polynomial estimator of degree at most $D$, and conditional Jensen's inequality gives
    \[
        \bbE\big[\|\overline{\btheta}(\bY^{(1)})-\btheta\|_2^2\big]
        \leq\bbE\big[\|\widehat{\btheta}(\bY^{(2)})-\btheta\|_2^2\big].
    \]
    Taking the infimum shows that the low-degree MMSE with diagonal variance two is no smaller than that with diagonal variance one. Thus the same lower bound holds, without changing $C_0$. 
\end{rmk}
\begin{rmk}[Dependence on the prior moments]\label{rmk:moment-priors}
    More generally, for any fixed nonzero-mean prior with all moments finite, put $M_D=\max\{1,\|\Theta\|_{4D+2}\}$. The proof gives the more precise estimate
    \begin{equation}\label{eq:main-moment-bound}
        \MMSE_{\le D}\geq m_2 - q_\amp-\sum_{g=1}^{D}\left( \frac{2^{54}(D+1)^{48}M_D^{16}}{n} \right)^g.
    \end{equation}
    Hence the AMP-risk lower bound holds whenever $(D(n)+1)^{48}M_{D(n)}^{16}=o(n)$. For a fixed sub-Gaussian prior, $M_D\leq2K\sqrt{D+1}$, so this estimate also allows $D=o(n^{1/56})$. For a fixed bounded-support prior, it allows $D=o(n^{1/48})$. These exponents are not optimized; the simpler formulation in Theorem~\ref{thm:main} uses the common exponent $60$. Under either of the latter two assumptions, the limiting equality follows when $D(n)\to\infty$ by the same fixed-degree upper bound. 
\end{rmk}
The proof extends the cumulant method of \cite{schramm2022computational} through a more selective application of Jensen's inequality. We introduce an auxiliary Gaussian observation $\bR$ whose strength is chosen at the AMP fixed point and use conditional joint cumulants to retain suitable signal information in the denominator of the low-degree correlation quotient. Conditioning on $\bR$ preserves the product structure of the signal distribution, producing cancellations that allow quantitative control as the degree grows. The tree contribution is controlled by Gaussian integration by parts and Hermite orthogonality. In each remaining core, scalar stability controls ordinary degree-two vertices, while an averaged posterior-cumulant estimate is needed only at a number of exceptional occurrences proportional to the core excess. We believe this method is broadly applicable for proving the equivalence between AMP and Low-Deg in Gaussian additive models. When $q_\bayes>q_\amp$ and the Bayes risk is characterized by the variational formula, the resulting lower bound exhibits a statistical-computational gap and cannot be deduced from a statistical lower bound alone; see, e.g., \cite{montanari2025equivalence, sohn2025sharp}.

\section{Proof overview}\label{sec:proof-overview}

We now explain the main idea of the proof. Throughout, we use only the standing assumptions that $\pi_\Theta$ has nonzero mean and all moments finite. We first establish the moment-dependent bound in Remark~\ref{rmk:moment-priors}; the sub-Gaussian assumption is used only in the final comparison with \eqref{eq:main-finite}. By exchangeability, the normalized vector MMSE equals the minimum mean-square error for estimating $\btheta_1$. Optimizing the scale of each scalar polynomial gives the variational formula
\begin{equation}\label{eq:correlation-variational}
    \MMSE_{\le D} =m_2-\sup_{\substack{\deg f\le D\\\bbE[f(\bY)^2]>0}} \left\{ \frac{\bbE[\btheta_1f(\bY)]^2}{\bbE[f(\bY)^2]} \right\},
\end{equation}
as in \cite[Fact~1.1]{schramm2022computational}. The difficulty lies in the denominator: the marginal law of $\bY$ is a mixture of Gaussian shifts, and a convenient orthogonal polynomial basis for this law is generally unavailable. Our starting point is the proof of \cite[Theorem~2.2]{schramm2022computational}, which avoids this difficulty by averaging over the signal before squaring. More precisely, their application of Jensen's inequality is
\begin{equation}\label{eq:jensen}
    \bbE_{\btheta,\bW}[f(\bY)^2] \ge\bbE_{\bW}\left[ \left(\bbE_{\btheta}  f\left(n^{-1/2}\btheta\btheta^{\top}+\bW\right) \right)^2\right];
\end{equation}
see \cite[Eq.~(7)]{schramm2022computational}. The conditional expectation on the right is a polynomial in the Gaussian noise, so it can be expanded in the Hermite basis. This leads to a bound in terms of unconditional joint cumulants. We first recall the notation needed to state this bound. Let $\mathcal E_n=\{(i,j):1\le i\le j\le n\}$, write $x_{ij}(\btheta)=\btheta_i\btheta_j$, and put $x=(x_e)_{e\in\mathcal E_n}$. For real random variables $B_1,\ldots,B_m$ with the required mixed moments finite, define their joint cumulant by
\begin{equation}\label{eq:cum-defn}
    \cum(B_1,\ldots,B_m)
    =\sum_{\pi\in\Pi([m])}(-1)^{|\pi|-1}(|\pi|-1)!
        \prod_{A\in\pi}\bbE\left[\prod_{j\in A}B_j\right],
\end{equation}
where $\Pi([m])$ is the set of partitions of $m$ labeled elements into nonempty unlabeled blocks, and $|\pi|$ is the number of blocks; see, e.g.,~\cite{novak2014three}. This finite partition definition requires no moment-generating-function assumption. Write $\bbN_0=\{0,1,2,\ldots\}$. For a multi-index $\alpha\in\bbN_{0}^{\mathcal E_n}$, put $|\alpha|=\sum_e\alpha_e$, $\alpha!=\prod_e\alpha_e!$, and $x^\alpha=\prod_e x_e^{\alpha_e}$. Write $\alpha\le\beta$ for coordinatewise comparison and $\binom{\beta}{\alpha}=\beta!/(\alpha!(\beta-\alpha)!)$ when $\alpha\le\beta$. Let $h_k$ be the orthonormal probabilists' Hermite polynomial, and set $h_\alpha(\bW)=\prod_e h_{\alpha_e}(\bW_e)$. Define
\[
    \cC_\alpha=\cum\big(
        \btheta_1,
        \underbrace{x_e(\btheta),\ldots,x_e(\btheta)}_{\alpha_e\text{ copies}}:
        e\in\mathcal E_n\big),
    \qquad \cC_0=\bbE[\btheta_1].
\]
The bound of \cite{schramm2022computational}, specialized to \eqref{eq:model-def}, is
\begin{equation}\label{eq:simple-bound-mmse}
    \MMSE_{\le D}
    \ge m_2-\sum_{|\alpha|\le D}
        \frac{\cC_\alpha^2}{n^{|\alpha|}\alpha!}.
\end{equation}
However, the bound in \eqref{eq:simple-bound-mmse} need not be sharp for the problem considered here. Precisely, the loss in \eqref{eq:jensen} is exactly
\[
    \bbE[f(\bY)^2] -\bbE\left[\bbE[f(\bY)\mid\bW]^2\right] =\bbE\left[\operatorname{Var}(f(\bY)\mid\bW)\right].
\]
Thus averaging out the signal can discard information relevant to the denominator. s that they work in a region where any \emph{non-trivial} estimator are believed to be computationally unavailable. Thus, we expect that any low-degree polynomial $f(\bY)$ only reveals negligible information of $\btheta$, so the equality in \eqref{eq:jensen} asymptotically holds. However, for the problem considered in this work, we are working in a regime that low-degree estimators can achieves non-trivial correlation with $\btheta$, so there will be a substantial loss if we apply \eqref{eq:jensen}. 

Our first conceptual input is to apply a variant of \eqref{eq:jensen}. For a sigma-algebra $\cG$, we can apply Jensen's inequality with respect to $\cG$:
\[
    \bbE[f(\bY)^2] \ge\bbE\left[\bbE[f(\bY)\mid\cG]^2\right].
\]
Whereas \cite{schramm2022computational} chooses $\cG=\sigma(\bW)$, we instead begin with $\cG=\sigma(\bR,\bW)$, where $\bR$ is an auxiliary observation of the signal which retains signal dependence and creates additional cancellation. Before specifying $\bR$, it is useful to discuss what we should expect from this auxiliary observation. First, it should reveal some information about $\btheta$, but not too much. Indeed, if $\bR$ is independent of $\btheta$, then we recover the unconditional bound \eqref{eq:simple-bound-mmse}. At the opposite extreme, taking $\bR=\btheta$ makes the Jensen step an equality, but the conditional-cumulant bound below becomes only $\MMSE_{\le D}\geq0$. Heuristically, $\bR$ should account for the part of the signal that low-degree estimators can already recover, leaving a tractable description of the remaining uncertainty. Second, conditioning on $\bR$ should preserve the independence of the signal coordinates. This ensures that conditional cumulants involving independent groups of coordinates still vanish, so that only connected configurations involving coordinate $1$ can contribute. 

The final choice of such a reference observation is motivated by the scalar Gaussian channels used in the interpolation arguments underlying the statistical characterization of this problem \cite{lelarge2017fundamental}. These arguments compare the coupled matrix observation with independent scalar observations of the form $\sqrt{q}\btheta_i+\bZ_i$. The scalar model is explicitly tractable, while Gaussian integration by parts relates it to the original model. The statistical variational formula selects the global minimizer $q_\bayes$ of $\Psi$. Here, however, we seek a lower bound at the AMP error level, which suggests using the same scalar-channel family at the smallest stable stationary point $q_\amp$, rather than at $q_\bayes$. Accordingly, the auxiliary observation used in the graph estimates is
\begin{equation}\label{eq:choice-bR}
    \bR=\sqrt{s}\btheta+\bZ, \qquad s=q_\amp,
\end{equation}
where $\bZ\sim\cN(0,I_n)$ is independent of $(\btheta,\bW)$. The conditional law of $\btheta$ given $\bR$ is then a product of scalar Gaussian-channel posterior distributions with prior $\pi_\Theta$. 

We first state the cumulant bounds for a general auxiliary observation. For a sigma-field $\mathcal G$, define the conditional joint cumulant by
\begin{equation}\label{eq:conditional-cum-def}
    \cum(B_1,\ldots,B_m\mid\mathcal G) =\sum_{\pi\in\Pi([m])}(-1)^{|\pi|-1}(|\pi|-1)! \prod_{A\in\pi}\bbE\left[\prod_{j\in A}B_j\mid\mathcal G\right].
\end{equation}
Let $\bR=\bR(\btheta,\bZ)\in\bbR^N$ be any auxiliary observation, where $\bZ$ is independent of $(\btheta,\bW)$, and put
\begin{equation}{\label{eq:def-cK}}
    \cK_\alpha(\bR)=\cum\big( \btheta_1, \underbrace{x_e(\btheta), \ldots, x_e(\btheta)}_{\alpha_e\text{ copies}}: e\in\mathcal E_n\mid\bR\big),  \qquad \cK_0(\bR)=\bbE[\btheta_1\mid\bR].
\end{equation}
The moment assumptions, the $L^p$ contraction property of conditional expectation, and H\"older's inequality imply that these conditional cumulants have finite moments of all orders. Applying Jensen's inequality with $\cG=\sigma(\bR,\bW)$ gives the following preliminary bound.
\begin{lem}\label{lem:conditional-mmse-weak}
    For every auxiliary observation $\bR$ as above,
    \begin{equation}\label{eq:conditional-mmse-weak}
        \MMSE_{\le D} \ge m_2-\sum_{|\alpha|\le D} \frac{\bbE[\cK_\alpha(\bR)^2]}{n^{|\alpha|}\alpha!}.
    \end{equation}
\end{lem}

Although it is tempting to use Lemma~\ref{lem:conditional-mmse-weak} for the remainder of the proof, the bound \eqref{eq:conditional-mmse-weak} still does not suffice for our purpose, as the coefficients in \eqref{eq:conditional-mmse-weak} still fluctuate with the auxiliary randomness after $\btheta$ is fixed. Our second step removes this fluctuation before applying Cauchy--Schwarz. The conditional moment-cumulant identity allows us to express the numerator in \eqref{eq:correlation-variational} using only the conditional cumulants averaged given $\btheta$. We therefore retain those averaged cumulants, together with the Gaussian noise, in the denominator. More precisely, set
\begin{equation}\label{eq:final-choice-cG}
    A_\alpha(\btheta)=\bbE[\cK_\alpha(\bR)\mid\btheta], \qquad \cG=\sigma\big(\bW,A_\alpha(\btheta):|\alpha|\le D\big).
\end{equation}
Applying Jensen's inequality with this choice of $\cG$ yields the bound used in the rest of the proof.
\begin{lem}\label{lem:bound-mmse}
    For every auxiliary observation $\bR=\bR(\btheta,\bZ)\in\bbR^N$, where $\bZ$ is independent of $(\btheta,\bW)$,
    \begin{equation}\label{eq:delicate-bound-mmse}
        \MMSE_{\le D} \ge m_2-\sum_{|\alpha|\le D} \frac{\bbE[A_\alpha(\btheta)^2]}{n^{|\alpha|}\alpha!}.
    \end{equation}
\end{lem}
\begin{rmk}
    Both Lemma~\ref{lem:conditional-mmse-weak} and Lemma~\ref{lem:bound-mmse} will be proved in \S\ref{sec:delicate-Jensen}. Appendix~\ref{sec:simple-proof-certificate} gives an alternative proof of Lemma~\ref{lem:bound-mmse} by constructing a low-degree certificate, following the duality approach in \cite[Proposition~1.3]{sohn2025sharp}. The Jensen proof makes explicit which signal information is retained in the denominator and how the conditional bound extends the argument of \cite{schramm2022computational}, so we present it in the main part of the paper.
\end{rmk}

For the same auxiliary observation, the improvement from \eqref{eq:conditional-mmse-weak} to \eqref{eq:delicate-bound-mmse} is exactly
\[
    \sum_{|\alpha|\le D} \frac{\bbE[\cK_\alpha(\bR)^2]-\bbE[A_\alpha(\btheta)^2]}{n^{|\alpha|}\alpha!}
    =\sum_{|\alpha|\le D} \frac{\bbE[\operatorname{Var}(\cK_\alpha(\bR)\mid\btheta)]}{n^{|\alpha|}\alpha!}\ge0,
\]
so we indeed get a sharper bound in Lemma~\ref{lem:bound-mmse}. We now explain why the Gaussian channel in \eqref{eq:choice-bR} gives the desired control. Two contributions require particular care: rooted trees, whose total weight need not be small, and long chains of degree-two vertices in the remaining graphs. The stationarity and stability conditions $\Psi'(q_\amp)=0$ and $\Psi''(q_\amp)\ge0$ imply
\[
    \bbE\left[\big(\bbE[\btheta_1\mid\bR]\big)^2\right]=q_\amp,
    \qquad
    \bbE\left[\operatorname{Var}(\btheta_1\mid\bR)^2\right]\le1,
\]
as proved in Lemma~\ref{lem:fixed-identity}. The first identity identifies the scalar-channel MMSE with the AMP error $m_2-q_\amp$ and gives the self-consistency relation $s=\bbE[(\bbE[\btheta_1\mid\bR])^2]$. Gaussian integration by parts and Hermite orthogonality turn this relation into a recursive bound on the entire rooted-tree contribution, including the zero-order term. The second inequality controls the factors contributed by ordinary degree-two vertices, preventing long chains from producing exponential growth. Together, these identities place the graph expansion at the AMP error level.

We now bound the right-hand side of \eqref{eq:delicate-bound-mmse} under this specific choice of $\bR$. From now on, we will regard $\alpha$ as a multigraph (with possibly multiple self-loops) with vertex labels in $[n]$, with each edge or self-loop $e=(i,j)$ occurring $\alpha_e$ times. Denote by $V(\alpha)=\cup_{(i,j)\in\cE_n:\,\alpha_{ij}>0}\{i,j\}$ its vertex set. Each edge copy contributes one to $|\alpha|$; in vertex degrees, each loop copy contributes two. The zero multi-index is handled separately. Also define $\tau(\alpha)=|\alpha|-|V(\alpha)|+1$ to be the excess of $\alpha$. We will split \eqref{eq:delicate-bound-mmse} into two parts. The first part consists of summations over $\alpha$'s that are simple trees that contain the vertex $1$, which contribute no more than $q_\amp$.
\begin{lem}\label{lem:contribution-rooted-tree}
    Let $\cT$ be the collection of nonempty connected $\alpha$ satisfying $|\alpha|=|V(\alpha)|-1$. For the auxiliary observation in \eqref{eq:choice-bR},
    \[
        \bbE[A_0(\btheta)^2] +\sum_{\substack{|\alpha|\le D,\ \alpha\in\cT\\1\in V(\alpha)}} \frac{\bbE[A_\alpha(\btheta)^2]}{n^{|\alpha|}\alpha!} \le q_\amp.
    \]
\end{lem}
We then bound the contribution from the remaining $\alpha$, showing that they contribute an $o(1)$ term in the stated degree range.
\begin{lem}\label{lem:contribution-non-principle}
    For the auxiliary observation in \eqref{eq:choice-bR}, the following hold.
    \begin{enumerate}
        \item[(1)] If $|\alpha|\ge1$ and either $1\notin V(\alpha)$ or $\alpha$ is disconnected, then $A_\alpha(\btheta)=0$ almost surely.
        \item[(2)] Let $\cJ$ be the collection of nonempty connected $\alpha$ with $\tau(\alpha)\ge1$. Then
        \[
            \sum_{\substack{|\alpha|\le D,\ \alpha\in\cJ\\1\in V(\alpha)}} \frac{\bbE[A_\alpha(\btheta)^2]}{n^{|\alpha|}\alpha!}
            \le\sum_{g=1}^{D}\left( \frac{2^{54}(D+1)^{48}M_D^{16}}{n} \right)^g.
        \]
    \end{enumerate}
\end{lem}
Substituting Lemmas~\ref{lem:contribution-rooted-tree} and \ref{lem:contribution-non-principle} into \eqref{eq:delicate-bound-mmse} proves \eqref{eq:main-moment-bound}. For a sub-Gaussian prior, $M_D\le2K\sqrt{D+1}$, and hence
\[
    2^{54}(D+1)^{48}M_D^{16} \le2^{70}K^{16}(D+1)^{56} \le[120K(D+1)]^{60}.
\]
This proves Theorem~\ref{thm:main}.

\section{Bounding the low-degree MMSE via conditional cumulants}\label{sec:delicate-Jensen}

We prove Lemmas~\ref{lem:conditional-mmse-weak} and \ref{lem:bound-mmse} for an arbitrary auxiliary observation $\bR=\bR(\btheta,\bZ)$, with $\bZ$ independent of $(\btheta,\bW)$. In particular, $\bW$ is independent of $(\btheta,\bR)$. We first establish two identities used in both proofs. Throughout this section, we identify a symmetric matrix with its upper-triangular coordinates, and all polynomial derivatives are taken with respect to those coordinates. Fix a scalar polynomial $f$ of degree at most $D$. With $\btheta$ held fixed, $f(\bY)$ is a polynomial of degree at most $D$ in the Gaussian noise, so its Hermite expansion is finite:
\begin{equation}\label{eq:noise-hermite-expansion}
    f(\bY)=\sum_{|\alpha|\le D}b_\alpha(\btheta)h_\alpha(\bW),
    \qquad
    b_\alpha(\btheta)=\bbE[f(\bY)h_\alpha(\bW)\mid\btheta].
\end{equation}
By Lemma~\ref{lem:int-by-part-Hermite}, 
\begin{equation}{\label{eq:Gaussian-int-by-part}}
    b_\alpha(\btheta) =\frac1{\sqrt{\alpha!}}\bbE[\partial^\alpha f(\bY)\mid\btheta].
\end{equation}
Each $b_\alpha$ is a polynomial in $\btheta$, and hence has finite moments of all orders. The variables $\cK_\alpha(\bR)$ and $A_\alpha(\btheta)$ also have finite moments of all orders, by the conditional partition formula and the moment assumptions. Thus all products and conditional expectations used below are integrable. In addition, Lemma~\ref{lem:conditional-moment-cumulant} implies that
\[
    \bbE[\btheta_1 x^\beta\mid\bR] =\sum_{\alpha\le\beta} \binom{\beta}{\alpha} \cK_\alpha(\bR)\bbE[ x^{\beta-\alpha}\mid\bR].
\]
Thus, for every polynomial $p$ of degree at most $D$,
\begin{equation}\label{eq:given-R-expectation}
    \bbE[\btheta_1p(x)\mid\bR] =\sum_{|\alpha|\le D}\frac{\cK_\alpha(\bR)}{\alpha!} \bbE[\partial^\alpha p(x)\mid\bR].
\end{equation}
Recall that $(x(\btheta))_e = \btheta_i \btheta_j$ for $e=(i,j) \in \cE_n$. Apply the identity \eqref{eq:given-R-expectation} to the polynomial $p(x)=\bbE_{\bW}[f(n^{-1/2}x+\bW)]$. The chain rule and \eqref{eq:Gaussian-int-by-part} yield
\[
    \partial^\alpha p(x(\btheta)) =n^{-|\alpha|/2}\bbE[\partial^\alpha f(\bY)\mid\btheta] =n^{-|\alpha|/2}\sqrt{\alpha!}\,b_\alpha(\btheta).
\]
Plugging this into \eqref{eq:given-R-expectation} and noting that $\bW$ is independent of $(\btheta,\bR)$, we obtain
\[
    \bbE[\btheta_1f(\bY)\mid\bR] = \bbE[ \btheta_1 p(x) \mid \bR ] =\sum_{|\alpha|\le D} \frac{\cK_\alpha(\bR)}{n^{|\alpha|/2}\sqrt{\alpha!}} \bbE[b_\alpha(\btheta)\mid\bR].
\]
Taking expectations gives the numerator identity
\begin{equation}\label{eq:conditional-numerator}
    \bbE[\btheta_1f(\bY)]
    =\sum_{|\alpha|\le D}
        \frac{\bbE[\cK_\alpha(\bR)
            \bbE[b_\alpha(\btheta)\mid\bR]]}
        {n^{|\alpha|/2}\sqrt{\alpha!}}.
\end{equation}

\begin{proof}[Proof of Lemma~\ref{lem:conditional-mmse-weak}]
    We retain $(\bR,\bW)$ in the denominator. By \eqref{eq:noise-hermite-expansion} and the independence of $\bW$ from $(\btheta,\bR)$,
    \[
        \bbE[f(\bY)\mid\bR,\bW] =\sum_{|\alpha|\le D} \bbE[b_\alpha(\btheta)\mid\bR]h_\alpha(\bW).
    \]
    Conditional Jensen's inequality and Hermite orthogonality therefore give
    \begin{align*}
        \bbE[f(\bY)^2] \ge\bbE\left[\bbE[f(\bY)\mid\bR,\bW]^2\right]
        =\sum_{|\alpha|\le D} \bbE\left[\bbE[b_\alpha(\btheta)\mid\bR]^2\right].
    \end{align*}
    Applying Cauchy--Schwarz to the finite sum of $L^2$ inner products in \eqref{eq:conditional-numerator} gives
    \begin{align*}
        \bbE[\btheta_1f(\bY)]^2
        &\le\left(\sum_{|\alpha|\le D}
            \frac{\bbE[\cK_\alpha(\bR)^2]}{n^{|\alpha|}\alpha!}\right)
            \left(\sum_{|\alpha|\le D}
                \bbE\left[\bbE[b_\alpha(\btheta)\mid\bR]^2\right]\right)\\
        &\le\left(\sum_{|\alpha|\le D}
            \frac{\bbE[\cK_\alpha(\bR)^2]}{n^{|\alpha|}\alpha!}\right)
            \bbE[f(\bY)^2].
    \end{align*}
    Since $f$ was arbitrary, \eqref{eq:correlation-variational} proves \eqref{eq:conditional-mmse-weak}.
\end{proof}

\begin{proof}[Proof of Lemma~\ref{lem:bound-mmse}]
    We first average out the auxiliary randomness in the numerator. For each $\alpha$, the defining property of conditional expectation gives
    \[
        \bbE\left[\cK_\alpha(\bR)\bbE[b_\alpha(\btheta)\mid\bR]\right] =\bbE[\cK_\alpha(\bR)b_\alpha(\btheta)] =\bbE[A_\alpha(\btheta)b_\alpha(\btheta)].
    \]
    The first equality uses the $\bR$-measurability of $\cK_\alpha(\bR)$, and the second uses the $\btheta$-measurability of $b_\alpha(\btheta)$ and the definition of $A_\alpha$. Substituting into \eqref{eq:conditional-numerator}, we obtain
    \begin{equation}\label{eq:averaged-cumulant-numerator}
        \bbE[\btheta_1f(\bY)] =\sum_{|\alpha|\le D} \frac{\bbE[A_\alpha(\btheta)b_\alpha(\btheta)]}{n^{|\alpha|/2}\sqrt{\alpha!}}.
    \end{equation}
    Let $\cG$ be the sigma-field in \eqref{eq:final-choice-cG}, and define
    \[
        \overline b_\alpha =\bbE\left[b_\alpha(\btheta) \mid (A_\beta(\btheta))_{|\beta|\le D}\right].
    \]
    For the fixed model, auxiliary channel, and degree $D$, the sigma-field $\cG$ does not depend on $f$. Since $\bW$ is independent of $\btheta$, conditioning $b_\alpha(\btheta)$ additionally on $\bW$ does not change $\overline b_\alpha$. Consequently, \eqref{eq:noise-hermite-expansion} gives
    \[
        \bbE[f(\bY)\mid\cG] =\sum_{|\alpha|\le D}\overline b_\alpha h_\alpha(\bW).
    \]
    The coefficients $\overline b_\alpha$ are functions of $\btheta$ and are independent of $\bW$. Conditional Jensen's inequality and Hermite orthogonality now imply
    \begin{equation}\label{eq:retained-cumulant-jensen}
        \bbE[f(\bY)^2] \ge\bbE\left[\bbE[f(\bY)\mid\cG]^2\right] =\sum_{|\alpha|\le D}\bbE[\overline b_\alpha^2].
    \end{equation}
    At the same time, the retained family contains $A_\alpha(\btheta)$, so
    \[
        \bbE[A_\alpha(\btheta)b_\alpha(\btheta)] =\bbE[A_\alpha(\btheta)\overline b_\alpha].
    \]
    Combining this equality with \eqref{eq:averaged-cumulant-numerator} and applying Cauchy--Schwarz yields
    \begin{align*}
        \bbE[\btheta_1f(\bY)]^2
        &\le\left(\sum_{|\alpha|\le D} \frac{\bbE[A_\alpha(\btheta)^2]}{n^{|\alpha|}\alpha!}\right) \left(\sum_{|\alpha|\le D}\bbE[\overline b_\alpha^2]\right)\\
        &\le\left(\sum_{|\alpha|\le D} \frac{\bbE[A_\alpha(\btheta)^2]}{n^{|\alpha|}\alpha!}\right) \bbE[f(\bY)^2],
    \end{align*}
    where the last inequality follows from \eqref{eq:retained-cumulant-jensen}. Taking the supremum in \eqref{eq:correlation-variational} proves \eqref{eq:delicate-bound-mmse}.
\end{proof}

\section{Contribution from the rooted tree part}

The goal of this section is to prove Lemma~\ref{lem:contribution-rooted-tree}. For $s>0$ and $h\in\bbR$, define
\begin{equation}\label{eq:scalar-posterior}
    Z_s(h)=\int_{\bbR}e^{hx-sx^2/2}\pi_\Theta(dx),\qquad
    \pi_{s,h}(dx)=\frac{e^{hx-sx^2/2}\pi_\Theta(dx)}{Z_s(h)}.
\end{equation}
Write $\eta_s(h)=\int x\,\pi_{s,h}(dx)=\partial_h\log Z_s(h)$. Bayes' formula and \eqref{eq:choice-bR} give
\begin{equation}{\label{eq:conditional-mean}}
    \mathcal L(\btheta\mid\bR)=\bigotimes_{i=1}^n\pi_{s,\bH_i},\qquad
    \bH_i=\sqrt{s}\bR_i=s\btheta_i+\sqrt{s}\bZ_i.
\end{equation}
Conditionally on $\bR$, the coordinates of $\btheta$ are independent of one another and of $\bW$. Throughout the graph argument, $s=q_\amp>0$ is fixed and we abbreviate $\eta=\eta_s$; all derivatives of $\eta$ are with respect to its field variable $h$. In particular, $\eta'(h)=\operatorname{Var}_{\pi_{s,h}}(\Theta)\geq0$. The quadratic tilt $e^{-sx^2/2}$ in \eqref{eq:scalar-posterior} is retained for a general prior.

For each $\alpha\in\bbN_{0}^{\cE_n}$ and $h \in \bbR^n$, let
\begin{equation}{\label{eq:def-bC-h}}
    \bC_\alpha(h)=\cum\big(\btheta_1, \underbrace{x_e(\btheta),\ldots,x_e(\btheta)}_{\alpha_e\text{ identical copies}}: e\in\mathcal E_n\big),
\end{equation}
where the cumulant is evaluated under the independent product law $\bigotimes_{i=1}^n\pi_{s,h_i}$. Then $\cK_\alpha(\bR)=\bC_\alpha(\bH)$ and $\bC_0(h)=\eta(h_1)$. Lemma~\ref{lem:posterior-regularity} shows that posterior moments, these cumulants, and their derivatives have polynomial growth in $h$ for fixed positive $s$. This justifies the differentiations and conditional Gaussian integrations by parts below. We first state a leaf-removal lemma which will be used repeatedly.
\begin{lem}\label{lem:leaf-remove}
    Suppose $v\neq1$ is a vertex of degree one in $\alpha$, and
    $(v,w)\in\alpha$. Write $\alpha-v$ for the multi-index obtained
    by deleting the unique edge incident to $v$. Then, for every
    $h\in\bbR^n$,
    \begin{equation}\label{eq:leaf-remove}
        \bC_\alpha(h)
        =\eta(h_v)\,\partial_{h_w}\bC_{\alpha-v}(h).
    \end{equation}
\end{lem}
\begin{proof}
    Set $m=|\alpha|\geq1$, and let $B_1,\ldots,B_m$ be the variables defining $\bC_{\alpha-v}(h)$: namely, $B_1=\btheta_1$, and $B_2,\ldots,B_m$ are the edge variables $x_e(\btheta)$ with multiplicities prescribed by $\alpha-v$. When $m=1$, the latter list is empty and $\bC_0(h)=\bbE[\btheta_1]=\eta(h_1)$. Thus $\bC_{\alpha-v}(h)=\cum(B_1,\ldots,B_m)$. Since $v\neq1$ has degree one, we have $w\neq v$, the edge $(v,w)$ occurs exactly once, and none of $B_1,\ldots,B_m$ depends on $\btheta_v$. Hence $\btheta_v$ is independent of $(\btheta_w,B_1,\ldots,B_m)$. In every term of the partition formula \eqref{eq:cum-defn} for $\cum(\btheta_v\btheta_w,B_1,\ldots,B_m)$, exactly one block contains the distinguished variable $\btheta_v\btheta_w$. If $A\subseteq[m]$ indexes the other variables in that block, independence gives
    \[
        \bbE\left[\btheta_v\btheta_w\prod_{j\in A}B_j\right] =\eta(h_v)\, \bbE\left[\btheta_w\prod_{j\in A}B_j\right],
    \]
    with the usual convention that an empty product equals one. The moments corresponding to the other blocks are unchanged. Factoring out $\eta(h_v)$ from every partition term yields
    \begin{equation}{\label{eq:leaf-factorization}}
        \bC_\alpha(h) =\cum(\btheta_v\btheta_w,B_1,\ldots,B_m)
        =\eta(h_v)\,\cum(\btheta_w,B_1,\ldots,B_m).
    \end{equation}
    It remains to show that the cumulant on the right equals $\partial_{h_w}\bC_{\alpha-v}(h)$. Differentiating the normalized measure in \eqref{eq:scalar-posterior} gives
    \[
        \partial_{h_w}\pi_{s,h_w}(dx)
        =\left(x-\frac{Z_s'(h_w)}{Z_s(h_w)}\right)\pi_{s,h_w}(dx)
        =(x-\eta(h_w))\pi_{s,h_w}(dx).
    \]
    Consequently, for every polynomial $F(\btheta)$ with no explicit dependence on $h$,
    \[
        \partial_{h_w}\bbE[F(\btheta)]
        =\bbE[\btheta_w F(\btheta)]-\eta(h_w)\bbE[F(\btheta)].
    \]
    For each $\pi\in\Pi([m])$, applying this identity to every block moment gives
    \begin{align*}
        \partial_{h_w}\prod_{A\in\pi}\bbE\left[\prod_{j\in A}B_j\right] =\sum_{A\in\pi} \bbE\left[\btheta_w\prod_{j\in A}B_j\right] \prod_{\substack{A'\in\pi\\A'\neq A}} \bbE\left[\prod_{j\in A'}B_j\right] -|\pi|\eta(h_w) \prod_{A\in\pi}\bbE\left[\prod_{j\in A}B_j\right].
    \end{align*}
    The factor $|\pi|$ in the second term appears because differentiating any one of the $|\pi|$ block moments contributes the same normalization term. Multiplying by $(-1)^{|\pi|-1}(|\pi|-1)!$ and summing over $\pi\in\Pi([m])$ gives 
    \begin{align}
        \partial_{h_w}\bC_{\alpha-v}(h) =\ & \sum_{ \pi\in\Pi([m]) } (-1)^{|\pi|-1} (|\pi|-1)! \sum_{A \in \pi} \bbE\left[\btheta_w\prod_{j\in A}B_j\right] \prod_{\substack{A'\in\pi\\A'\neq A}} \bbE\left[\prod_{j\in A'}B_j\right] \nonumber \\
        &+ \eta(h_w) \sum_{ \pi\in\Pi([m]) } (-1)^{|\pi|} (|\pi|)! \prod_{A\in\pi}\bbE\left[\prod_{j\in A}B_j\right]. \label{eq:partial-equivalence-1}
    \end{align}
    Since $\bbE[\btheta_w]=\eta(h_w)$, \eqref{eq:partial-equivalence-1} can be further written as
    \begin{align}
        \partial_{h_w}\bC_{\alpha-v}(h) =\ & \sum_{ \pi\in\Pi([m]) } (-1)^{|\pi|-1} (|\pi|-1)! \sum_{A \in \pi} \bbE\left[\btheta_w\prod_{j\in A}B_j\right] \prod_{\substack{A'\in\pi\\A'\neq A}} \bbE\left[\prod_{j\in A'}B_j\right] \nonumber \\
        &+  \sum_{ \pi\in\Pi([m]) } (-1)^{|\pi|} (|\pi|)! \mathbb E\left[ \btheta_w \right] \prod_{A\in\pi}\bbE\left[\prod_{j\in A}B_j\right]. \label{eq:partial-equivalence-2}
    \end{align}
    Label the additional occurrence of $\btheta_w$ by $\star$, distinct from the labels $1,\ldots,m$ of $B_1,\ldots,B_m$. Every partition of $\{\star\}\cup[m]$ is obtained uniquely from a partition $\pi\in\Pi([m])$ either by replacing one block $A\in\pi$ with $A\cup\{\star\}$, or by adjoining the singleton block $\{\star\}$. The first case is counted by the first sum in \eqref{eq:partial-equivalence-2}; the second has one extra block and is counted by the second sum. Thus, \eqref{eq:partial-equivalence-2} yields
    \[
        \partial_{h_w}\bC_{\alpha-v}(h)
        =\partial_{h_w}\cum(B_1,\ldots,B_m)
        =\cum(\btheta_w,B_1,\ldots,B_m).
    \]
    For $m=1$, this reduces to $\partial_{h_w}\bbE[B_1]=\operatorname{Cov}(\btheta_w,B_1)$. Substituting this identity for the cumulant factor in \eqref{eq:leaf-factorization}, while retaining the prefactor $\eta(h_v)$, proves \eqref{eq:leaf-remove}.
\end{proof}

Now we can finish the proof of Lemma~\ref{lem:contribution-rooted-tree}.
\begin{proof}[Proof of Lemma~\ref{lem:contribution-rooted-tree}]
    Root each nonempty tree $\alpha$ at vertex $1$, and let $d_i(\alpha)$ be the number of children of $i$ in this rooted tree. Repeated application of \eqref{eq:leaf-remove} gives
    \[
        \bC_\alpha(h)=\prod_{i\in V(\alpha)}\eta^{(d_i(\alpha))}(h_i),
    \]
    where $\eta^{(0)}=\eta$. Since $\cK_\alpha(\bR)=\bC_\alpha(\bH)$ and the coordinates of $\bH$ are independent conditionally on $\btheta$, the product formula implies
    \[
        A_\alpha(\btheta) =\prod_{i\in V(\alpha)} \bbE[\eta^{(d_i(\alpha))}(\bH_i)\mid\btheta_i].
    \]
    The coordinates of $\btheta$ are independent and identically distributed, so 
    \[
        \bbE[A_\alpha(\btheta)^2]=\prod_{i\in V(\alpha)}a_{d_i(\alpha)}, \qquad a_k=\bbE[(\bbE[\eta^{(k)}(\bH_1)\mid\btheta_1])^2].
    \] 
    Also, $a_0=\bbE[A_0(\btheta)^2]$, and $\alpha!=1$ for every nonempty tree $\alpha$.

    We next bound a finite sum of the coefficients $a_k$. For each $k\geq0$, Gaussian integration by parts $k$ times, conditionally on $\btheta_1$, gives
    \[
        \bbE[\eta(\bH_1)h_k(\bZ_1)\mid\btheta_1] =\frac{s^{k/2}}{\sqrt{k!}}
        \bbE[\eta^{(k)}(\bH_1)\mid\btheta_1].
    \]
    Since $\bZ_1$ is standard Gaussian independently of $\btheta_1$, Bessel's inequality for the orthonormal Hermite polynomials yields, for every integer $m\geq0$,
    \[
        \sum_{k=0}^{m}\frac{s^k}{k!} \left(\bbE[\eta^{(k)}(\bH_1)\mid\btheta_1]\right)^2 \leq\bbE[\eta(\bH_1)^2\mid\btheta_1].
    \]
    Taking expectations and using Lemma~\ref{lem:fixed-identity}(1), we obtain
    \begin{equation}{\label{eq:a-k-inequality}}
        \sum_{k=0}^{m}s^ka_k/k!\leq q_\amp.
    \end{equation}
    It remains to sum the tree weights. For every integer $m\geq0$, put
    \begin{equation}{\label{eq:def-S-m}}
        S_m=a_0+ \sum_{\substack{\alpha\in\cT,\,1\in V(\alpha)\\
                    1\leq|\alpha|\leq m}}
        \left(\frac{1}{n}\right)^{|\alpha|}
        \prod_{i\in V(\alpha)}a_{d_i(\alpha)}.
    \end{equation}
    The term $a_0$ represents the tree consisting only of its root. We will show by induction that $S_m\leq q_\amp$. The case $m=0$ follows by applying \eqref{eq:a-k-inequality} with $m=0$. Fix $m\geq1$. A nonempty tree counted by $S_m$ has some $k\in\{1,\ldots,\min(m,n-1)\}$ children at its root. Deleting the root leaves $k$ rooted subtrees, each having at most $m-1$ edges. The original weight is the product of their weights, multiplied by $a_k(1/n)^k$. Ordering the $k$ child subtrees counts each original labeled tree exactly $k!$ times. We obtain an upper bound by dropping the requirements that the subtree vertex sets be pairwise disjoint, that they avoid the original root, and that their total number of edges be at most $m-k$. For a prescribed root label, the sum of the weights of trees with at most $m-1$ edges is $S_{m-1}$: the product weights are invariant under relabeling. Allowing this root label to be any of the $n$ labels thus gives total weight $nS_{m-1}$ for each child subtree. All weights are nonnegative, and hence
    \begin{align*}
        S_m
        &\leq a_0+ \sum_{k=1}^{\min(m,n-1)}\frac{a_k}{k!} \left(\frac{1}{n}\right)^k (nS_{m-1})^k\\
        &\leq\sum_{k=0}^{m}\frac{a_k}{k!} (S_{m-1})^k \leq\sum_{k=0}^{m}\frac{a_k}{k!} (q_\amp)^k \leq q_\amp.
    \end{align*}
    The third inequality uses the induction hypothesis, and the last one uses $s=q_\amp$ and \eqref{eq:a-k-inequality}. This closes the induction. Taking $m=D$ and recalling the product formula gives the asserted inequality.
\end{proof}

\section{Contribution from the disconnected part}

The goal of this section is to prove Lemma~\ref{lem:contribution-non-principle}(1).

\begin{proof}[Proof of Lemma~\ref{lem:contribution-non-principle}(1)]
    Fix $\alpha$ with $|\alpha|\geq1$. Recall that $\cK_\alpha(\bR)= \bC_\alpha(\bH)$. We will show that $\bC_\alpha(h)=0$ for every $h\in\bbR^n$ under either of the two conditions in the statement. Now fix $h\in\bbR^n$, and take the coordinates of $\btheta$ to be independent with $\btheta_i\sim\pi_{s,h_i}$, as in the definition of $\bC_\alpha(h)$. Suppose first that $1\notin V(\alpha)$. Then, $\btheta_1$ is independent of the entire family of edge variables appearing in this cumulant, including their repeated copies. This family is nonempty because $|\alpha|\geq1$. By Lemma~\ref{lem:facts-cumulant}, $\bC_\alpha(h)=0$.

    It remains to consider the case where $1\in V(\alpha)$ and $\alpha$ is disconnected. Choose a connected component of $\alpha$ not containing $1$. This component contains at least one edge, possibly a self-loop. All remaining arguments of $\bC_\alpha(h)$, including $\btheta_1$, depend only on coordinates indexed outside this component. By Lemma~\ref{lem:facts-cumulant}(1), again $\bC_\alpha(h)=0$. 
    
    We have proved $\bC_\alpha(h)=0$ for every $h$ whenever either condition in the statement holds. Substituting $h=\bH$ gives $\cK_\alpha(\bR)=0$ almost surely, and hence
    \[
        A_\alpha(\btheta) =\bbE[\cK_\alpha(\bR)\mid\btheta] =0 \qquad\text{almost surely}.
    \]
    This proves the assertion.
\end{proof}

\section{Contribution from the non-tree part}

The goal of this section is to prove Lemma~\ref{lem:contribution-non-principle}(2). For $\alpha\in\cJ$ and $1 \in V(\alpha)$, we root $\alpha$ at $1$, and repeatedly delete degree-one vertices other than the root. As $\tau(\alpha)\ge 1$ and $\alpha$ is connected, the resulting rooted core $\gamma=\core(\alpha)$ is connected, all its nonroot vertices have degree at least two, and the root has degree at least one. All deleted edges form ordinary simple trees attached to vertices of $\gamma$. In addition, $\tau(\gamma)=\tau(\alpha)\geq1$. We first reduce the contribution of non-tree graphs to the contribution of cores.
\begin{lem}\label{lem:core-reduction}
    For every $n,D$,
    \begin{equation}\label{eq:core-reduction}
        \sum_{\substack{0<|\alpha|\le D\\ \alpha\in\cJ,\,1\in V(\alpha)}}
        \frac{1}{n^{|\alpha|}\alpha!} \bbE\left[A_\alpha(\btheta)^2\right] \le
        \sum_{\substack{\gamma\text{ rooted core}\\ 1\le|\gamma|\le D}} \frac{n^{-\tau(\gamma)}}{\aut(\gamma)\gamma!} \bbE\left[\bC_\gamma(\bH)^2\right].
    \end{equation}
    Here the sum over rooted cores contains one representative of each rooted isomorphism class with at most $n$ vertices. We write $\aut(\gamma)$ for the number of vertex automorphisms of $\gamma$ that fix the root and preserve all edge multiplicities, including loop multiplicities.
\end{lem}
\begin{proof}
    Recall from the proof of Lemma~\ref{lem:contribution-rooted-tree} the coefficients $a_k=\bbE[(\bbE[\eta^{(k)}(\bH_1)\mid\btheta_1])^2]$. First fix a particular labeled core $\gamma$ on a subset of $[n]$, with root $1$ and $1\le|\gamma|\le D$. Consider $\alpha\in\cJ$ with $\core(\alpha)=\gamma$ and $|\alpha|\le D$. The edges outside $\gamma$ form simple trees, each attached to exactly one vertex of $\gamma$. Root each tree at the vertex in $\gamma$, and let $d_i(\alpha)$ denote the number of children of $i$ in this forest. For a core vertex, this counts only its neighbors outside the core. Put $k=(d_i(\alpha))_{i\in V(\gamma)}$. Let $r_{i,\ell}$ be the $\ell$th child attached to $i\in V(\gamma)$, and let $T_{i,\ell}$ be the subtree rooted at $r_{i,\ell}$, excluding $i$ and the attaching edge. Then we can write
    \begin{equation}{\label{eq:alpha-decomposition}}
        \alpha = \gamma \cup \left( \bigcup_{i \in V(\gamma)} \bigcup_{1 \le \ell \le d_i(\alpha)}
        \big(\{(i,r_{i,\ell})\}\cup T_{i,\ell}\big) \right).
    \end{equation}
    A subtree $T_{i,\ell}$ may consist of a single vertex; throughout, $|T|$ denotes the number of its edges. For a multi-index $k\in\bbN_{0}^{V(\gamma)}$, write 
    \[
        |k|=\sum_{i\in V(\gamma)}k_i, \qquad k!=\prod_{i\in V(\gamma)}k_i!, \qquad \partial^k=\prod_{i\in V(\gamma)}\partial_{h_i}^{k_i}.
    \]
    Repeatedly applying Lemma~\ref{lem:leaf-remove} gives
    \[
        \bC_\alpha(h) =(\partial^k\bC_\gamma)(h) \prod_{i\in V(\alpha)\setminus V(\gamma)} \eta^{(d_i(\alpha))}(h_i).
    \]
    Using $A_\alpha(\btheta)=\bbE[\bC_\alpha(\bH)\mid\btheta]$ and the conditional independence of the coordinates of $\bH$, we obtain
    \[
        A_\alpha(\btheta) =\bbE[(\partial^k\bC_\gamma)(\bH)\mid\btheta] \prod_{i\in V(\alpha)\setminus V(\gamma)} \bbE[\eta^{(d_i(\alpha))}(\bH_i)\mid\btheta_i].
    \]
    The first factor depends only on the coordinates of $\btheta$ indexed by $V(\gamma)$. Independence of the coordinates of $\btheta$ therefore implies
    \[
        \bbE[A_\alpha(\btheta)^2] = \bbE\left[ \left(\bbE[(\partial^k\bC_\gamma)(\bH) \mid\btheta]\right)^2 \right] \prod_{i\in V(\alpha)\setminus V(\gamma)}a_{d_i(\alpha)}.
    \]
    Also, $\alpha!=\gamma!$, since every edge outside the core has
    multiplicity one. 
    Thus,
    \begin{align}
        & \sum_{\substack{\alpha\in\cJ,\ |\alpha|\leq D\\ \core(\alpha)=\gamma}}
             \frac{1}{n^{|\alpha|}\alpha!}\bbE[A_\alpha(\btheta)^2] \nonumber\\
        =\ &\frac{1}{n^{|\gamma|}\gamma!}
          \sum_{\substack{k\in\bbN_0^{V(\gamma)}\\|k|\leq D-|\gamma|}}
          \bbE\left[\left(\bbE[(\partial^k\bC_\gamma)(\bH)\mid\btheta]\right)^2\right]
          \sum_{\substack{\alpha\in\cJ,\ |\alpha|\leq D,\ \core(\alpha)=\gamma\\
                         (d_i(\alpha))_{i\in V(\gamma)}=k}}
        n^{-(|\alpha|-|\gamma|)}
        \prod_{i\in V(\alpha)\setminus V(\gamma)}a_{d_i(\alpha)} \nonumber \\
        \le\ & \frac{1}{n^{|\gamma|}\gamma!}
          \sum_{\substack{k\in\bbN_0^{V(\gamma)}\\|k|\leq D-|\gamma|}}
          \bbE\left[\left(\bbE[(\partial^k\bC_\gamma)(\bH)\mid\btheta]\right)^2\right] \sum_{ \substack{ \{ T_{i,\ell}: i \in V(\gamma), 1 \le \ell \le k_i \} \\  |T_{i,\ell}| \le D } } \omega(T_{i,\ell}), .
          \label{eq:reduce-to-core-0}
    \end{align}
    where for a rooted simple tree $T$ we put (below $d_j(T)$ denotes the number of children of $j$ in $T$)
    \[
        \omega(T) = (1/n)^{|T|+1}\prod_{j \in V(T)} a_{d_j(T)}.
    \]
    Here the $+1$ term in the exponent of $n$ in the definition of $\omega(T)$ comes from the contribution of the additional edge $(i,r_{i,\ell})$ in \eqref{eq:alpha-decomposition}, and the last inequality of \eqref{eq:reduce-to-core-0} follows from from the fact that all weights $\omega(T)$ are nonnegative, so we obtain an upper bound by allowing the trees to intersect one another and the core, and by replacing the total edge constraint with the requirement that each branch tree separately have at most $D$ edges.
    
    We now sum over the attached trees $\{ T_{i,\ell}: i \in V(\gamma), 1 \le \ell \le k_i \}$ while keeping $k$ fixed in \eqref{eq:reduce-to-core-0}. For each $i\in V(\gamma)$, order the $k_i$ branches attached to $i$. Every such graph is then counted exactly $\prod_i k_i!=k!$ times.
    \[
        \sum_{ \substack{ \{ T_{i,\ell}: i \in V(\gamma), 1 \le \ell \le k_i \} \\  |T_{i,\ell}| \le D } } \omega(T_{i,\ell}) \leq\frac{1}{k!}
        \prod_{i\in V(\gamma)}\prod_{\ell=1}^{k_i}
        \left(\frac1n\sum_{r=1}^n
             \sum_{\substack{T:\,\operatorname{root}(T)=r\\|T|\leq D}}
                  \omega(T)\right).
    \]
    Recall \eqref{eq:def-S-m}.For every prescribed root label $r$, the inner tree sum equals $S_D$, including the singleton tree. Hence each parenthesized expression equals $n^{-1}nS_D=S_D$, and
    \[
        \sum_{ \substack{ \{ T_{i,\ell}: i \in V(\gamma), 1 \le \ell \le k_i \} \\  |T_{i,\ell}| \le D } } \omega(T_{i,\ell}) \leq\frac{(S_D)^{|k|}}{k!}.
    \]
    By the proof of Lemma~\ref{lem:contribution-rooted-tree}, $S_D\le q_\amp=s$. Since $|k|\le D-|\gamma|$ for every graph in the original sum, it follows that
    \begin{align}
        &\sum_{\substack{\alpha\in\cJ,\,|\alpha|\le D\\
                         \core(\alpha)=\gamma}}
        \frac{1}{n^{|\alpha|}\alpha!}
        \bbE[A_\alpha(\btheta)^2] \nonumber \\
        \le\ &
        \frac{1}{n^{|\gamma|}\gamma!}
        \sum_{\substack{k\in\bbN_{0}^{V(\gamma)}\\
                         |k|\le D-|\gamma|}}
        \frac{(S_D)^{|k|}}{k!}
        \bbE\left[
            \left(\bbE[(\partial^k\bC_\gamma)(\bH)
                \mid\btheta]\right)^2
        \right] \nonumber \\
        \le\ &
        \frac{1}{n^{|\gamma|}\gamma!}
        \sum_{\substack{k\in\bbN_{0}^{V(\gamma)}\\
                         |k|\le D-|\gamma|}}
        \frac{s^{|k|}}{k!}
        \bbE\left[
            \left(\bbE[(\partial^k\bC_\gamma)(\bH)
                \mid\btheta]\right)^2
        \right]. \label{eq:reduce-to-core-1}
    \end{align}
    To bound \eqref{eq:reduce-to-core-1}, apply the multivariate version of the Hermite argument used in the proof of Lemma~\ref{lem:contribution-rooted-tree}: Conditionally on $\btheta$, we have $\bH_i=s\btheta_i+\sqrt{s}\bZ_i$, with independent standard Gaussian coordinates $\bZ_i$. Thus, applying Gaussian integration by parts in every core coordinate, for each $k\in\bbN_{0}^{V(\gamma)}$,
    \[
        \bbE\left[
            \bC_\gamma(\bH)
            \prod_{i\in V(\gamma)}h_{k_i}(\bZ_i)
            \mathrel{\Big|}\btheta
        \right]
        =\frac{s^{|k|/2}}{\sqrt{k!}}
        \bbE[(\partial^k\bC_\gamma)(\bH)\mid\btheta].
    \]
    The products $\prod_{i\in V(\gamma)}h_{k_i}(\bZ_i)$ form an orthonormal family conditionally on $\btheta$. Bessel's inequality gives
    \[
        \sum_{\substack{k\in\bbN_{0}^{V(\gamma)}\\
                         |k|\le D-|\gamma|}}
        \frac{s^{|k|}}{k!}
        \left(\bbE[(\partial^k\bC_\gamma)(\bH)
            \mid\btheta]\right)^2
        \le\bbE[\bC_\gamma(\bH)^2\mid\btheta].
    \]
    Taking expectations and substituting into \eqref{eq:reduce-to-core-1} proves that, for each fixed labeled core,
    \[
        \sum_{\substack{\alpha\in\cJ,\,|\alpha|\le D\\
                         \core(\alpha)=\gamma}}
        \frac{1}{n^{|\alpha|}\alpha!}
        \bbE[A_\alpha(\btheta)^2]
        \le
        \frac{1}{n^{|\gamma|}\gamma!}
        \bbE[\bC_\gamma(\bH)^2].
    \]

    Finally, for a rooted core $\gamma$, the number of distinct labeled copies on subsets of $[n]$ with root $1$ is 
    \[
        \frac{(n-1)!}{(n-|V(\gamma)|)!\aut(\gamma)} \le n^{|V(\gamma)|-1}/\aut(\gamma);
    \]
    the factorial ratio before division by $\aut(\gamma)$ counts injections of the nonroot vertices into $[n]\setminus\{1\}$. Hence the total contribution of graphs whose cores belong to this rooted isomorphism class is at most
    \[
        \frac{n^{|V(\gamma)|-1}}{\aut(\gamma)}
        \frac{1}{n^{|\gamma|}\gamma!}
        \bbE[\bC_\gamma(\bH)^2]
        =\frac{n^{-\tau(\gamma)}}
              {\aut(\gamma)\gamma!}
        \bbE[\bC_\gamma(\bH)^2].
    \]
    Summing over the rooted isomorphism classes with
    $1\le|\gamma|\le D$ proves \eqref{eq:core-reduction}.
\end{proof}

We then record the following result on $\bC_{\gamma}(h)$ defined in \eqref{eq:def-bC-h}. Put $e=|\gamma|$ and label every edge copy, including repeated edges and loops. For $i\in V(\gamma)$, let $\cI_i$ be the set of all tagged occurrences of $\btheta_i$ in these edge variables and the additional $\btheta_1$. Both endpoint occurrences of a self-loop are distinct tagged elements, as are occurrences in different copies of the same edge. Thus $|\cI_i|=\deg_\gamma(i)+\mathbf1_{\{i=1\}}$ and $\sum_{i\in V(\gamma)}|\cI_i|=2e+1$. Choose a partition $\pi_i\in\Pi(\cI_i)$ for each vertex, and write $\pi=(\pi_i)_{i\in V(\gamma)}$. Associate a multigraph to $\pi$ as follows. Its vertices are the blocks of all the partitions $\pi_i$. Each edge copy of $\gamma$ joins the two blocks containing its two endpoint occurrences. The block containing the distinguished occurrence of $\btheta_1$ is the root and is retained even if it is isolated. Call $\pi$ connected when this multigraph is connected.

\begin{figure}[htbp]
\centering
\includegraphics[width=0.8\linewidth]{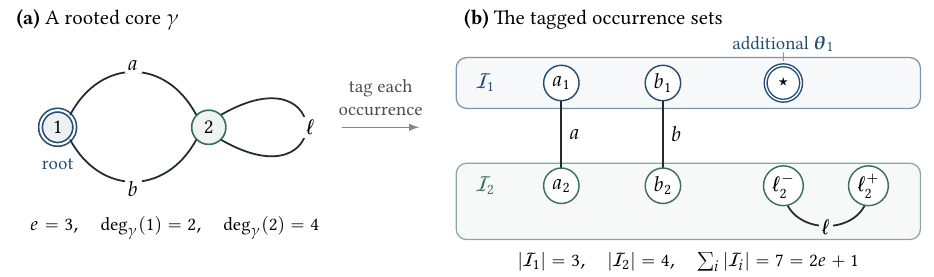}
\caption{Tagged occurrences for a rooted core $\gamma$ with edge copies
$a,b$ joining $1$ to $2$ and a loop $\ell$ at $2$.
The labels $a_i,b_i$ denote the endpoint occurrences of
$\boldsymbol\theta_i$ in the corresponding edge copies, while $\star$
denotes the additional occurrence of $\boldsymbol\theta_1$.
The loop tags $\ell_2^-,\ell_2^+$ are distinct occurrences of the same
variable $\boldsymbol\theta_2$. Thus
$\mathcal I_1=\{\star,a_1,b_1\}$ and
$\mathcal I_2=\{a_2,b_2,\ell_2^-,\ell_2^+\}$.}
\label{fig:tagged-occurrence-sets}
\end{figure}

\begin{figure}[htbp]
\centering
\includegraphics[width=0.8\linewidth]{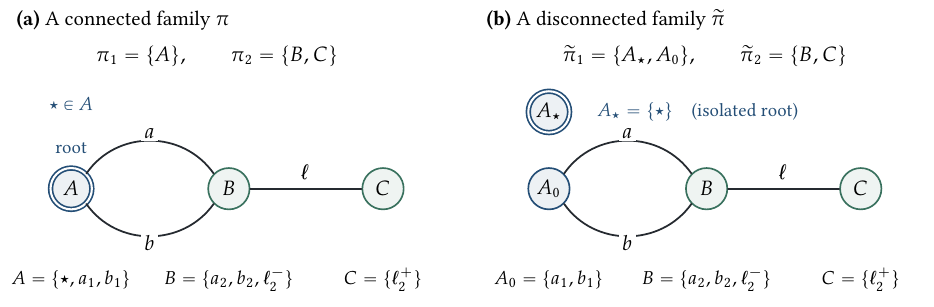}
\caption{The multigraphs associated with two partition families of the
sets in Figure~\ref{fig:tagged-occurrence-sets}.
Each block is a vertex, and each edge copy joins the blocks containing
its endpoint tags. In~(a), the family $\pi$ gives a connected multigraph
rooted at $A$. The edges $a,b$ remain parallel, whereas the loop $\ell$
becomes an ordinary edge because $\ell_2^-$ and $\ell_2^+$ lie in
different blocks. In~(b), splitting $A$ into $A_\star=\{\star\}$ and
$A_0=\{a_1,b_1\}$ leaves an isolated root, which is retained.}
\label{fig:partition-multigraph}
\end{figure}

\begin{lem}{\label{lem:bC-h-calculus}}
    Suppose that $\btheta_i\sim\pi_{s,h_i}$ independently, with fixed $s>0$ and $\eta=\eta_s$. For any $m\geq1$, the cumulant of $m$ occurrences of $\btheta_i$ is
    \begin{equation}{\label{eq:partition-formula}}
        \cum\big( \underbrace{\btheta_i,\ldots,\btheta_i}_{m\text{ copies}} \big)
        =\left.\frac{d^m}{dz^m}\log\frac{Z_s(h_i+z)}{Z_s(h_i)}\right|_{z=0}
        =\eta^{(m-1)}(h_i).
    \end{equation}
    In addition, with $\bC_\gamma(h)$ defined as in \eqref{eq:def-bC-h}, we have for every core $\gamma$
    \begin{equation}{\label{eq:one-core-exact-form}}
        \bC_\gamma(h) =\sum_{\substack{ \pi_i\in\Pi(\cI_i),\ i\in V(\gamma)\\ \pi\text{ connected}}} \prod_{i\in V(\gamma)} \prod_{A\in\pi_i}\eta^{(|A|-1)}(h_i).
    \end{equation}
\end{lem}
\begin{proof}
    Under $\pi_{s,h_i}$, the moment generating function of $\btheta_i$ is $Z_s(h_i+z)/Z_s(h_i)$, so \eqref{eq:partition-formula} follows from $\eta=\partial_h\log Z_s$. We now prove \eqref{eq:one-core-exact-form}. Fix a core $\gamma$ and put $e=|\gamma|$. Label its $e$ edge copies, including repeated copies and self-loops, and let $B_1=\btheta_1$ and $B_2,\ldots,B_{e+1}$ be the corresponding edge variables. Thus $\bC_\gamma(h)=\cum(B_1,\ldots,B_{e+1})$. 
    By \eqref{eq:cum-defn},
    \begin{align*}
        \bC_\gamma(h) &=\sum_{\sigma\in\Pi([e+1])} (-1)^{|\sigma|-1}(|\sigma|-1)! \prod_{A\in\sigma} \bbE\left[\prod_{j\in A}B_j\right] \\
        &= \sum_{\sigma\in\Pi([e+1])} (-1)^{|\sigma|-1}(|\sigma|-1)! \prod_{A\in\sigma} \bbE\left[\prod_{o \in \varphi(A)} \btheta_{i(o)} \right] \\
        &= \sum_{\sigma\in\Pi([e+1])} (-1)^{|\sigma|-1}(|\sigma|-1)! \prod_{A\in\sigma} \left( \sum_{ \kappa_A \in \Pi(\varphi(A)) } \prod_{ A' \in \kappa_A } \cum(\btheta_{i(o)}:o \in A') \right).
    \end{align*}
    Here $\varphi(A)$ is the disjoint union of the tagged occurrences in the arguments $B_j$ with $j\in A$, and $i(o)$ is the original vertex of occurrence $o$. In particular, $\varphi(A)$ is not a set of distinct vertex labels. The third equality is the inverse moment--cumulant identity, obtained by expanding the exponential of the cumulant generating function and comparing the coefficient in which each tagged variable occurs once.
    A cumulant block $A'$ containing occurrences from two or more different vertices vanishes by independence and Lemma~\ref{lem:facts-cumulant}(1). Hence every term $(\sigma,\kappa_A:A \in \sigma)$ that is not forced to vanish in this way induces a family $\pi=(\pi_i)_{i\in V(\gamma)}$ where $\pi_i\in\Pi(\cI_i)$ (note that different $\sigma$ may induce the same family $\pi$). 
    
    Call an outer partition $\sigma\in\Pi([e+1])$ compatible with $\pi$ if every block of every $\pi_i$ is contained in $\varphi(A)$ for some $A\in\sigma$. For a fixed compatible pair $(\sigma,\pi)$, $(\kappa_A:A \in \sigma)$ is uniquely determined: $\kappa_A$ consists of exactly the blocks of $\pi$ contained in $\varphi(A)$. By \eqref{eq:partition-formula}, the product of its cumulant factors is $\prod_{i\in V(\gamma)}\prod_{A\in\pi_i}
    \eta^{(|A|-1)}(h_i)$, independently of $\sigma$. Reordering
    the finite sums gives
    \[
        \bC_\gamma(h)
        =\sum_{\substack{\pi_i\in\Pi(\cI_i)\,,\ i\in V(\gamma)}}
            \left(
                \sum_{\substack{\sigma\in\Pi([e+1])\\
                    \sigma\text{ compatible with }\pi}}
                    (-1)^{|\sigma|-1}(|\sigma|-1)!
            \right)
            \prod_{i\in V(\gamma)}
            \prod_{A\in\pi_i}\eta^{(|A|-1)}(h_i).
    \]

    Fix $\pi$ and let $c$ be the number of connected components
    of its associated multigraph. Each index in $[e+1]$ belongs
    to a unique component: index $1$ belongs to the component
    containing the distinguished root occurrence, and each other
    index belongs to the component containing its edge copy.
    The components thus induce a partition of $[e+1]$ into $c$ nonempty parts. We claim that $\sigma$ is compatible with $\pi$ precisely when
    it groups these $c$ parts without splitting any of them.
    To prove the forward implication, observe that two edge
    copies incident to the same block of $\pi$ have indices in
    the same block of $\sigma$, by compatibility. The same holds
    for index $1$ and any edge incident to the block containing
    the root occurrence. Following paths in the associated
    multigraph therefore shows that all indices belonging to
    any one component lie in a single block of $\sigma$.
    Conversely, suppose that $\sigma$ does not split the indices
    of any component. All occurrences in a fixed block of $\pi$
    belong to arguments whose indices lie in that block's
    component: the edge arguments are incident to the block,
    and the root argument belongs to it when present. These
    indices consequently lie in a single block of $\sigma$,
    which is exactly compatibility.

    It follows that compatible outer partitions are in bijection
    with partitions of the $c$ components, with the number of
    blocks preserved. The coefficient of the product associated
    with $\pi$ is therefore
    \[
        \sum_{\sigma\in\Pi([c])}
            (-1)^{|\sigma|-1}(|\sigma|-1)!
        =\cum\big(\underbrace{1,\ldots,1}_{c\text{ copies}}\big)
        =\mathbf1_{\{c=1\}}.
    \]
    This proves \eqref{eq:one-core-exact-form}.
\end{proof}

We now bound each core-structured cumulant using the following lemma.
\begin{lem}{\label{lem:one-core-bound}}
    Fix a rooted core $\gamma$ with $e=|\gamma|\leq D$, $v=|V(\gamma)|$, and $g=\tau(\gamma)$. Then
    \begin{equation}\label{eq:one-core}
        \bbE[\bC_\gamma(\bH)^2]
        \leq(D+1)^{2g}(2D+1)^{36g}2^{8g}M_D^{16g},
    \end{equation}
    where $M_D=\max\{1,\|\Theta\|_{4D+2}\}$ as in Remark~\ref{rmk:moment-priors}.
\end{lem}
\begin{proof}
    We will use the exact expansion of $\bC_\gamma(h)$ in \eqref{eq:one-core-exact-form}. 
    We first show that
    \begin{equation}{\label{eq:num-summands-bound}}
        \#\{ \pi_i\in\Pi(\cI_i),\ i\in V(\gamma), \pi\text{ connected} \} \le (D+1)^g(2D+1)^{6g}.
    \end{equation}
    Since $\gamma$ is a rooted core, every nonroot vertex has degree at least two and the root has degree at least one. Hence $|\cI_i|\geq2$ for every vertex. Put $\cB=\{i\in V(\gamma):|\cI_i|\geq3\}$. The identity $\sum_{i\in V(\gamma)}(|\cI_i|-2)=2e+1-2v=2g-1$ implies
    \[
        |\cB|\leq2g-1, \qquad \sum_{i\in\cB}|\cI_i| =2g-1+2|\cB|\leq6g-3\leq6g.
    \]
    For a connected family $\pi$, its associated multigraph has $e$ edges and $\sum_i|\pi_i|$ vertices. Connectedness therefore gives $\sum_{i\in V(\gamma)}(|\pi_i|-1)\leq e+1-v=g$. In particular, if $\cS=\{i\in V(\gamma)\setminus\cB:|\pi_i|=2\}$, then $|\cS|\leq g$. At each vertex outside $\cB$, there are only two possible partitions: the partition with one block of size two, and the partition with two singleton blocks. Thus their partitions are specified completely by $\cS$. Since $v=e+1-g\leq D$, the number of choices of $\cS$ is at most $\sum_{k=0}^{\min(g,v)}\binom{v}{k}\leq(v+1)^g\leq(D+1)^g$. A set of size $m$ has at most $m^m$ partitions. Therefore, the number of choices of $(\pi_i)_{i\in\cB}$ is at most
    \[
        \prod_{i\in\cB}|\cI_i|^{|\cI_i|} \leq(2D+1)^{\sum_{i\in\cB}|\cI_i|} \leq(2D+1)^{6g}.
    \]
    Equation~\eqref{eq:num-summands-bound} then follows from the multiplication principle.

    It remains to bound each summand in $L^2$. Fix a connected family $\pi$. At a vertex $i\in\cB$, the factors in \eqref{eq:one-core-exact-form} are conditional cumulants of $\btheta_i$ given $\bH_i$, with total order $|\cI_i|$. Lemma~\ref{lem:posterior-cumulant-moments}, applied with $\mathcal G=\sigma(\bH_i)$, and $2|\cI_i|\leq4D+2$ give
    \[
        \left\|\prod_{A\in\pi_i}\eta^{(|A|-1)}(\bH_i)\right\|_2
        \leq|\cI_i|^{2|\cI_i|}\|\Theta\|_{2|\cI_i|}^{\,|\cI_i|}
        \leq(2D+1)^{2|\cI_i|}M_D^{|\cI_i|}.
    \]
    Multiplying these bounds over $\cB$ costs at most $(2D+1)^{12g}M_D^{6g}$. At a vertex $i\in\cS$, the factor is $\eta(\bH_i)^2$, and conditional Jensen yields
    \[
        \|\eta(\bH_i)^2\|_2
        =\left(\bbE[\eta(\bH_i)^4]\right)^{1/2}
        \leq\|\Theta\|_4^2\leq M_D^2.
    \]
    At every vertex outside $\cB\cup\cS$, the factor is exactly $\eta'(\bH_i)$. Such a vertex has $|\cI_i|=2$: it is either a nonroot vertex of degree two or the root of degree one. We retain the sharper scalar stability bound from Lemma~\ref{lem:fixed-identity}(2) at these vertices, rather than estimating them by prior moments.

    Under the original law, the coordinates of $\bH$ are independent and identically distributed. Thus the $L^2$ norm of each summand factors over vertices, and
    \begin{align}
        \left\|
            \prod_{i\in V(\gamma)}\prod_{A\in\pi_i}
                \eta^{(|A|-1)}(\bH_i)
        \right\|_2
        &\leq(2D+1)^{12g}M_D^{8g}
            \left\{\bbE[(\eta'(\bH_1))^2]\right\}^{(v-|\cB|-|\cS|)/2}
            \nonumber\\
        &\leq(2D+1)^{12g}M_D^{8g}.
        \label{eq:single-l2-bound}
    \end{align}
    The first inequality uses $|\cS|\leq g$ and $M_D\geq1$, and the second follows from scalar stability. We retain the larger factor $2^{8g}$ in \eqref{eq:one-core} to keep the final constants uniform.

    Finally, applying the triangle inequality to \eqref{eq:one-core-exact-form}, and combining \eqref{eq:single-l2-bound} with \eqref{eq:num-summands-bound}, gives
    \[
        \left\{\bbE[\bC_\gamma(\bH)^2]\right\}^{1/2}
        \leq(D+1)^g(2D+1)^{18g}2^{4g}M_D^{8g}.
    \]
    Squaring proves \eqref{eq:one-core}. In particular, the high-moment cost is proportional to the core excess $g$, not to the number of edges in its degree-two chains.
\end{proof}

We are now able to finish the proof of Lemma~\ref{lem:contribution-non-principle}(2).
\begin{proof}[Proof of Lemma~\ref{lem:contribution-non-principle}(2)]
    By Lemmas~\ref{lem:core-reduction} and \ref{lem:one-core-bound}, 
    \begin{align}
        &\sum_{\substack{\alpha\in\cJ,\ 1\in V(\alpha)\\
                         |\alpha|\le D}}
            \frac{1}{n^{|\alpha|}\alpha!}
            \bbE[A_\alpha(\btheta)^2] \le
        \sum_{\substack{\gamma\text{ rooted core}\\
                        1\le|\gamma|\le D}}
        \frac{n^{-\tau(\gamma)}}
             {\aut(\gamma)\gamma!}
        \bbE\left[\bC_\gamma(\bH)^2\right] \nonumber \\
        \le\ &
            \sum_{g=1}^{D}\frac{(D+1)^{2g}(2D+1)^{36g}2^{8g}M_D^{16g}}{n^g}
            \sum_{\substack{\gamma\text{ rooted core}\\
                            1\le|\gamma|\le D,\ \tau(\gamma)=g}}
            \frac{1}{\aut(\gamma)\gamma!}.  \label{eq:non-principle-relax-1}
    \end{align}
    It remains to count rooted cores with a prescribed excess. Fix $1\le g\le D$, and consider a rooted core $\gamma$ with $1\le|\gamma|\le D$ and $\tau(\gamma)=g$. Every nonroot vertex has degree at least two, while the root has degree at least one. The degree-sum identity therefore gives
    \[
        \sum_{i\in V(\gamma)\setminus\{1\}}
            (\deg_\gamma(i)-2)
        =2g-\deg_\gamma(1)\le 2g-1.
    \]
    Consequently, there are at most $2g-1$ nonroot vertices with degree at least three. Suppress all nonroot vertices of degree two: delete such a vertex and replace its two incident edges by a single edge joining their other endpoints, allowing the new edge to be a self-loop. A nonroot degree-two vertex cannot be incident only to a self-loop, since it would then form a connected component not containing the root. Each suppression preserves connectedness and the degrees of the remaining vertices, and decreases both the number of edges and the number of vertices by one. Thus it preserves the excess. The remaining vertices are precisely the root and the nonroot vertices whose original degrees were at least three. If their number is $m$, then $1\le m\le 2g$, and the resulting multigraph has $m+g-1\le 3g-1$ edges. For each edge of the resulting multigraph, record the number of original edges that it represents. Conversely, the resulting rooted multigraph together with these lengths determines $\gamma$ up to rooted isomorphism. To count these descriptions for a given $m$, label the remaining vertices by $[m]$, with the root labeled $1$, and order the $m+g-1$ edges. Each edge is specified by an ordered pair of endpoints in $[m]$ and a length in $[D]$. Hence there are at most $(m^2D)^{m+g-1}$ descriptions. Since $\aut(\gamma)\gamma!\ge1$, we conclude that
    \[
        \sum_{\substack{\gamma\text{ rooted core}\\
                        1\le|\gamma|\le D,\ \tau(\gamma)=g}}
            \frac{1}{\aut(\gamma)\gamma!}
        \le \sum_{m=1}^{2g}(m^2D)^{m+g-1} \le 2g\big((2g)^2D\big)^{3g-1}
        \le [2(D+1)]^{10g},
    \]
    where the last inequality uses $1\leq g\leq D$. Combining this estimate with \eqref{eq:non-principle-relax-1} and $2D+1\leq2(D+1)$ gives
    \begin{align*}
        \eqref{eq:non-principle-relax-1}
        &\leq\sum_{g=1}^{D}
            \frac{[2(D+1)]^{10g}(D+1)^{2g}(2D+1)^{36g}
                2^{8g}M_D^{16g}}{n^g}\\
        &\leq\sum_{g=1}^{D}\left(
            \frac{2^{54}(D+1)^{48}M_D^{16}}{n}
        \right)^g.
    \end{align*}
    This proves Lemma~\ref{lem:contribution-non-principle}(2), and hence \eqref{eq:main-moment-bound}. 
\end{proof}

\appendix

\section{Preliminary results}

In this section, we state the inputs that the proof will require.

\subsection{Hermite polynomials}

In this subsection, we record the following well-known property of Hermite polynomials. See \cite[Proposition~C.1]{montanari2025equivalence} for its proof.
\begin{lem}{\label{lem:Hermite-expansion}}
    The orthonormal probabilists' Hermite polynomials satisfy
    \[
        h_k(\mu+z) =\sum_{j=0}^k \frac{\sqrt{k!}}{\sqrt{j!}(k-j)!}\, z^{k-j} h_j(\mu).
    \]    
\end{lem}
We will also use the following generalization of a well-known integration by parts formula for Gaussian random variables. See \cite[Proposition~2.10]{kunisky2019notes} for its proof, after dividing their probabilists' Hermite polynomial by $\sqrt{k!}$.
\begin{lem}{\label{lem:int-by-part-Hermite}}
    If $f:\bbR\to\bbR$ is $k$ times continuously differentiable and $f(y)$ and its first $k$ derivatives are bounded by $\exp(O(|y|^{\alpha}))$ for some $\alpha\in(0,2)$, then
    \[
        \bbE_{y \sim \cN(0,1)}\left[ h_k(y)f(y) \right] = \frac{1}{\sqrt{k!}}\bbE_{y \sim \cN(0,1)}\left[ \partial^k f(y) \right].
    \]
\end{lem}

\subsection{Joint cumulants and conditional joint cumulants}

In this subsection, we record some properties of the joint cumulants and conditional joint cumulants, defined in \eqref{eq:cum-defn} and \eqref{eq:conditional-cum-def} respectively. The following lemma is proved in \cite[Propositions~2.11--2.13]{schramm2022computational}. 
\begin{lem}{\label{lem:facts-cumulant}}
    The following results hold whenever the required mixed moments are finite:
    \begin{enumerate}
        \item[(1)] If $a,b \ge 1$ and $X_1, X_2,\ldots,X_a,Y_1,\ldots,Y_b$ are random variables with $\{X_i\}_{i \in [a]}$ independent from $\{Y_j\}_{j \in [b]}$, then
        \[
            \cum(X_1,X_2,\ldots,X_a,Y_1,\ldots,Y_b) = 0.
        \]
        \item[(2)] If $X_1,\ldots,X_n$ and $Y_1,\ldots,Y_n$ are random variables with $\{X_i\}_{i\in[n]}$ independent from $\{Y_i\}_{i \in [n]}$, then
        \[
            \cum(X_1+Y_1,\ldots,X_n+Y_n) = \cum(X_1,\ldots,X_n) + \cum(Y_1,\ldots,Y_n).
        \]
        \item[(3)] The joint cumulant of order at least two is invariant under constant shifts; cumulants of every order are scaled by constant multiplication. That is, if $X_1,\ldots,X_n$ are jointly distributed random variables and $c$ is any constant, then
        \[
             \cum(X_1 + c, X_2, \ldots,X_n) = \cum(X_1,\ldots,X_n) + c \cdot \mathbf 1_{ \{ n=1 \} },
        \]
        and
        \[ 
        \cum(c X_1, X_2, \ldots, X_n) = c \cdot \cum(X_1,\ldots,X_n).
        \]
    \end{enumerate}
\end{lem}

We next record the conditional version of the moment--cumulant recurrence in \cite[Eq.~(9) and Claim~2.14]{schramm2022computational}. We give the partition argument explicitly to explain the binomial coefficient and to justify conditioning under the standing moment assumptions.
\begin{lem}\label{lem:conditional-moment-cumulant}
    Recall $x=(x_e(\btheta))_{e\in\cE_n}$ and the conditional cumulants $\cK_\alpha(\bR)$ defined in \eqref{eq:def-cK}. For every $\beta\in\bbN_{0}^{\cE_n}$, almost surely,
    \begin{equation}\label{eq:conditional-moment-cumulant}
        \bbE[\btheta_1 x^\beta\mid\bR] =\sum_{\alpha\le\beta} \binom{\beta}{\alpha} \cK_\alpha(\bR)\bbE[ x^{\beta-\alpha}\mid\bR].
    \end{equation}
\end{lem}
\begin{proof}
    Fix $\beta$ and put $m=|\beta|$. Set $B_0=\btheta_1$, and let $B_1,\ldots,B_m$ be a list containing $\beta_e$ labeled occurrences of $x_e(\btheta)$ for each $e\in\cE_n$. Thus $\prod_{j=1}^m B_j=x^\beta$. The labels distinguish occurrences of the same random variable; they do not represent independent copies. Every $B_j$ is a polynomial in $\btheta$, so all mixed moments needed below are finite. Moreover, the $L^p$ contraction property of conditional expectation and H\"older's inequality show that the conditional moments, conditional cumulants, and their finite products below are integrable. We use the inverse moment--cumulant relation; see \cite[\S1.3]{novak2014three}. In the present conditional setting, for every $I\subseteq\{0,\ldots,m\}$ it reads
    \[
        \bbE\left[\prod_{j\in I}B_j \mid \bR\right] =\sum_{\pi\in\Pi(I)}\prod_{A\in\pi} \cum(B_j:j\in A\mid\bR),
    \]
    where $\Pi(I)$ denotes the set of partitions of $I$. 
    This conditional form follows because moment--cumulant inversion is a finite polynomial identity in the mixed moments. Replacing them by conditional moments preserves the identity, with conditional cumulants defined by \eqref{eq:conditional-cum-def}. For the fixed finite family of variables, it holds outside a single null set.
    We take $\Pi(\varnothing)$ to consist of the empty partition and every empty product to equal one, so the identity also holds when $I=\varnothing$.

    Apply this relation with $I=\{0,\ldots,m\}$. Every partition of $I$ has a unique block containing $0$. Write that block as $\{0\}\cup J$, where $J\subseteq[m]$. Once $J$ is fixed, the remaining blocks form an arbitrary partition of $[m]\setminus J$. Hence the total contribution of all partitions having this distinguished block is
    \begin{align*}
        &\cum(B_0,B_j:j\in J\mid\bR) \sum_{\pi\in\Pi([m]\setminus J)}\prod_{A\in\pi} \cum(B_j:j\in A\mid\bR) \\
        =\ & \cum(B_0,B_j:j\in J\mid\bR) \bbE[\prod_{j\in[m]\setminus J}B_j\mid\bR],
    \end{align*}
    where the equality follows from applying the same moment--cumulant relation to $[m]\setminus J$. Summing over all choices of $J$ therefore gives
    \[
        \bbE[\btheta_1x^\beta\mid\bR] =\sum_{J\subseteq[m]} \cum(B_0,B_j:j\in J\mid\bR) \bbE\left[\prod_{j\in[m]\setminus J}B_j\mid\bR\right].
    \]

    We now group these subsets according to the multiplicities of their edge variables. For each $e\in\cE_n$, let $\alpha_e$ be the number of labeled occurrences of $x_e(\btheta)$ selected by $J$. Then $\alpha\le\beta$, and the definition \eqref{eq:def-cK} gives
    \[
        \cum(B_0,B_j:j\in J\mid\bR)=\cK_\alpha(\bR),
        \qquad
        \prod_{j\in[m]\setminus J}B_j=x^{\beta-\alpha}.
    \]
    For a fixed $\alpha\le\beta$, there are $\binom{\beta_e}{\alpha_e}$ choices of the selected labels belonging to edge $e$. The choices for different edges specify $J$ uniquely, so the number of subsets with these multiplicities is $\prod_{e\in\cE_n}\binom{\beta_e}{\alpha_e}=\binom{\beta}{\alpha}$. Thus all such subsets have the same contribution $\cK_\alpha(\bR)\bbE[x^{\beta-\alpha}\mid\bR]$, and grouping them proves \eqref{eq:conditional-moment-cumulant}. When $\alpha=0$, the distinguished block is the singleton $\{0\}$ and its cumulant is $\cK_0(\bR)=\bbE[\btheta_1\mid\bR]$. When $\alpha=\beta$, the remaining partition is empty and contributes one. These conventions also cover $\beta=0$, when \eqref{eq:conditional-moment-cumulant} is the defining identity for $\cK_0$. No conditional independence among the edge variables was used, so the argument applies to every auxiliary observation $\bR$ allowed in the statement.
\end{proof}

Finally, we record an averaged estimate for products of conditional cumulants. This replaces a uniform bound on the posterior moments and therefore applies also to unbounded priors.
\begin{lem}\label{lem:posterior-cumulant-moments}
    Let $\mathcal G$ be a sigma-algebra, and let $m_1,\ldots,m_\ell\geq1$ be integers, with $r=\sum_{a=1}^{\ell}m_a$. Assume $\|\Theta\|_{2r}<\infty$. Then
    \[
        \left\|\prod_{a=1}^{\ell}
            \cum\big(\underbrace{\Theta,\ldots,\Theta}_{m_a\text{ copies}}\mid\mathcal G\big)
        \right\|_2
        \leq r^{2r}\|\Theta\|_{2r}^{\,r}.
    \]
\end{lem}
\begin{proof}
    For every $m\geq1$, the conditional partition formula and H\"older's inequality give
    \[
        \left|\cum\big(\underbrace{\Theta,\ldots,\Theta}_{m\text{ copies}}\mid\mathcal G\big)\right|
        \leq\sum_{\sigma\in\Pi([m])}(|\sigma|-1)!
            \bbE[|\Theta|^m\mid\mathcal G]
        \leq m^{2m}\bbE[|\Theta|^m\mid\mathcal G].
    \]
    Indeed, each product of conditional moments is bounded by $\bbE[|\Theta|^m\mid\mathcal G]$, there are at most $m^m$ partitions, and every coefficient is at most $m^m$. Applying conditional H\"older once more, and using $\prod_a m_a^{2m_a}\leq r^{2r}$, bounds the absolute product in the statement by $r^{2r}\bbE[|\Theta|^r\mid\mathcal G]$. Conditional Jensen and the tower property give
    \[
        \left\|\bbE[|\Theta|^r\mid\mathcal G]\right\|_2
        \leq(\bbE[|\Theta|^{2r}])^{1/2}=\|\Theta\|_{2r}^{\,r},
    \]
    proving the claim.
\end{proof}

\subsection{Regularity of the scalar posterior}

The posterior in \eqref{eq:scalar-posterior} is smooth at positive channel strength for every probability distribution $\pi_\Theta$. The following elementary estimate justifies the differentiations and integrations by parts in the proof.
\begin{lem}\label{lem:posterior-regularity}
    Fix $0<s_0<s_1<\infty$. For every integer $r\geq1$, there is a constant $C=C(r,s_0,s_1,\pi_\Theta)$ such that
    \[
        \int |x|^r\,\pi_{s,h}(dx)\leq C(1+|h|)^r,
        \qquad s\in[s_0,s_1],\quad h\in\bbR.
    \]
    Every fixed mixed derivative in $s$ and $h$ of a posterior moment or cumulant has at most polynomial growth in $h$, locally uniformly for $s>0$. The same holds for the finite graph cumulants $\bC_\alpha(h)$ and their field derivatives.
\end{lem}
\begin{proof}
    Choose $L<\infty$ such that $p=\pi_\Theta([-L,L])>0$. Then
    \[
        Z_s(h)\geq p\exp(-L|h|-s_1L^2/2),\qquad s\in[s_0,s_1].
    \]
    Choose $A=A(s_0,s_1,L)$ sufficiently large. If $|x|\geq A(1+|h|)$, then
    \[
        hx-sx^2/2+L|h|+s_1L^2/2\leq-s_0x^2/8.
    \]
    Splitting the posterior integral into $|x|<A(1+|h|)$ and its complement yields
    \[
        \int |x|^r\,\pi_{s,h}(dx)
        \leq A^r(1+|h|)^r
            +p^{-1}\sup_{x\in\bbR}|x|^r e^{-s_0x^2/8},
    \]
    which proves the first assertion.

    Differentiation of the defining integrals is justified locally by the Gaussian damping in \eqref{eq:scalar-posterior}. For every fixed polynomial $f$,
    \begin{align*}
        \partial_h\int f(x)\,\pi_{s,h}(dx)
        &=\int xf(x)\,\pi_{s,h}(dx)
            -\eta_s(h)\int f(x)\,\pi_{s,h}(dx),\\
        \partial_s\int f(x)\,\pi_{s,h}(dx)
        &=-\frac12\left\{
            \int x^2f(x)\,\pi_{s,h}(dx)
            -\left(\int x^2\,\pi_{s,h}(dx)\right)
             \left(\int f(x)\,\pi_{s,h}(dx)\right)
        \right\}.
    \end{align*}
    Iterating these identities expresses every derivative as a finite polynomial in posterior moments. The first assertion therefore gives polynomial growth. The cumulant partition formula \eqref{eq:cum-defn}, together with independence across coordinates, gives the same conclusion for posterior cumulants and $\bC_\alpha$.
\end{proof}
Conditionally on a fixed $\btheta$, the field $\bH=s\btheta+\sqrt s\bZ$ is Gaussian. The lemma ensures that boundary terms vanish in every fixed-order integration by parts. The finite prior moments then justify averaging over $\btheta$, as well as differentiating scalar-channel expectations locally for $s>0$. 

\subsection{The Bayes AMP value}

We verify the two scalar estimates used in the graph expansion at the auxiliary channel strength $s=q_\amp$: a self-consistency relation for the posterior mean and an averaged stability bound for the posterior variance. Throughout this subsection, the prior has nonzero mean and all moments finite.

For an arbitrary $s>0$, let $H=s\Theta+\sqrt{s}Z$, where $\Theta\sim\pi_\Theta$ and $Z\sim\cN(0,1)$ are independent, and set
\[
    F(s)=\bbE[\eta_s(H)^2],\qquad F(0)=(\bbE[\Theta])^2.
\]
Here $\eta_s$ and $\pi_{s,h}$ are the posterior mean and posterior law from \eqref{eq:scalar-posterior}. Since multiplication of the observation $\sqrt{s}\Theta+Z$ by $\sqrt{s}$ is invertible for $s>0$, conditioning on either observation gives the same posterior. Thus, in our field normalization,
\[
    \bbE[\Theta\mid H]=\eta_s(H),\qquad
    \operatorname{Var}(\Theta\mid H)=\eta_s'(H).
\]
The prime on $\eta_s$ denotes differentiation with respect to $h$, with $s$ held fixed. The conditional-variance identity gives
\[
    m_2-F(s)=\bbE[\operatorname{Var}(\Theta\mid H)]
        =\bbE[\eta_s'(H)].
\]
Expectations are taken over $(\Theta,Z)$. Conditional Jensen's inequality gives $(\bbE[\Theta])^2\leq F(s)\leq m_2$.

The I--MMSE identity~\cite[Theorem~1]{guo2005mutual} and the scalar MMSE derivative formula~\cite[Proposition~9]{guo2010estimation} give, with mutual information measured in nats,
\begin{equation}\label{eq:scalar-channel-derivatives}
    I_{\pi_\Theta}'(s)=\frac12\big(m_2-F(s)\big),\qquad
    F'(s)=\bbE\big[\eta_s'(s\Theta+\sqrt{s}Z)^2\big],
    \qquad s>0.
\end{equation}
The regularity theorem in~\cite[Proposition~7]{guo2010estimation} implies that $F$ is smooth on $(0,\infty)$ and continuous at zero with the value specified above. The first identity extends to the right derivative at zero. These results apply under our standing moment assumptions. In contrast to the field derivative $\eta_s'$, the derivative $F'(s)$ accounts for the dependence of both the posterior mean and the observation law on $s$.

\begin{lem}{\label{lem:fixed-identity}}
    Suppose $\bbE[\Theta]\neq0$ and all moments of $\Theta$ are finite. The infimum defining $q_\amp$ in \eqref{eq:def-q-bayes-amp} is attained, and
    \[
        0<(\bbE[\Theta])^2\leq q_\amp\leq m_2.
    \]
    Moreover, $q_\amp$ is the smallest nonnegative fixed point of $F$. Recalling \eqref{eq:conditional-mean}, and putting $s=q_\amp$ and $\eta=\eta_s$, we have:
    \begin{enumerate}
        \item[(1)] $\bbE[\eta(\bH_1)^2]=q_\amp$.
        \item[(2)] $\bbE[(\eta'(\bH_1))^2]\leq1$.
    \end{enumerate}
\end{lem}
\begin{proof}
    Substituting \eqref{eq:scalar-channel-derivatives} into the definition of $\Psi$ in \eqref{eq:mutual-info} gives
    \[
        \Psi'(q)=\frac12\big(q-F(q)\big),\qquad
        \Psi''(q)=\frac12\big(1-F'(q)\big),\qquad q>0,
    \]
    and $\Psi'(0)=-\frac12(\bbE[\Theta])^2<0$, where the derivative at zero is understood from the right. Consequently, positive stationary points of $\Psi$ are exactly the fixed points of $F$. We first show that a smallest such point exists, and then check that it satisfies the second-derivative condition in the definition of $q_\amp$.

    The continuous function $q\mapsto q-F(q)$ is negative at zero and nonnegative at $m_2$, since $F(m_2)\leq m_2$. It therefore has a zero in $(0,m_2]$. Its zero set in $[0,m_2]$ is nonempty and compact, and continuity at zero keeps this set a positive distance from zero. Let $q_*$ be its smallest element. There are no fixed points above $m_2$, because $F(q)\leq m_2$ for every $q$. Also,
    \[
        q_*=F(q_*)\geq(\bbE[\Theta])^2.
    \]
    By minimality, $q-F(q)<0$ throughout $[0,q_*)$: otherwise continuity and the strict negativity at zero would give an earlier zero. Hence $\Psi'(q)<0$ for $0\leq q<q_*$ and $\Psi'(q_*)=0$. Since $q_*>0$ and $\Psi'$ is differentiable there,
    \[
        \Psi''(q_*)
        =\lim_{\varepsilon\downarrow0}
            \frac{\Psi'(q_*)-\Psi'(q_*-\varepsilon)}{\varepsilon}
        \geq0.
    \]
    Thus $q_*$ belongs to the set defining $q_\amp$, while no smaller point is stationary. This proves that the infimum is attained at $q_*$ and that $q_\amp=q_*$. Neither uniqueness of the fixed point nor strict positivity of $\Psi''(q_\amp)$ is required.

    At $s=q_\amp$, the field $H$ has the same law as $\bH_1$ in \eqref{eq:conditional-mean}. The fixed-point equation $F(q_\amp)=q_\amp$ proves~(1). The derivative identity in \eqref{eq:scalar-channel-derivatives} and the inequality just established give
    \[
        \bbE[(\eta'(\bH_1))^2]
        =F'(q_\amp)=1-2\Psi''(q_\amp)\leq1,
    \]
    which proves~(2).
\end{proof}

\begin{rmk}{\label{rmk:application-AMP-value}}
    At the auxiliary channel strength $s=q_\amp$, \eqref{eq:conditional-mean} translates the lemma into
    \begin{align*}
        \bbE\left[\big(\bbE[\btheta_1\mid\bR]\big)^2\right]=q_\amp,\qquad \bbE[\operatorname{Var}(\btheta_1\mid\bR)]=m_2-q_\amp,\qquad \bbE[\operatorname{Var}(\btheta_1\mid\bR)^2]\leq1.
    \end{align*}
    The first identity calibrates the rooted-tree bound in Lemma~\ref{lem:contribution-rooted-tree}. The last inequality controls the factors at distinct ordinary degree-two vertices in Lemma~\ref{lem:one-core-bound}: independence of the coordinates of $\bH$ under the joint signal-and-noise law gives a product of $L^2$ norms, each at most one. Thus long chains cause no growth.
\end{rmk}

\section{An alternative proof via a low-degree certificate}\label{sec:simple-proof-certificate}

We give an alternative proof of Lemma~\ref{lem:bound-mmse} using the certificate formulation of \cite[Proposition~1.3]{sohn2025sharp}. The following elementary duality bound reduces the problem to constructing a suitable random variable with small $L^2$ norm.
\begin{lem}{\label{lem:duality}}
    We say a square-integrable random variable $u=u(\btheta,\bW)$ is a \emph{degree-$D$ certificate} for $\btheta_1$, if 
    \begin{equation}{\label{eq:certificate}}
        \mathbb E\left[ uf(\bY) \right] = \mathbb E[ \btheta_1f(\bY) ] \quad\text{for every scalar polynomial $f$ of degree at most $D$.}
    \end{equation}
    We then have $\MMSE_{\le D} \ge m_2-\bbE[u^2]$ for any degree-$D$ certificate $u$.
\end{lem}
\begin{proof}
    The proof is essentially the same as the linear algebra duality argument in \cite[Proposition~1.3]{sohn2025sharp}, but we record the argument here for completeness. For any degree-$D$ certificate $u$ and any scalar polynomial $f$ of degree at most $D$,
    \[
        \bbE\left[ (f(\bY)-\btheta_1)^2 \right] =\bbE\left[ (f(\bY)-u)^2 \right] +m_2-\bbE[u^2] \geq m_2-\bbE[u^2].
    \]
    By symmetry, we have $\bbE[(f_i(\bY)-\btheta_i)^2] \ge m_2-\bbE[u^2]$ for all $1 \le i \le n$. Taking the infimum over $\widehat{\btheta}=(f_1(\bY),\ldots,f_n(\bY))$ yields the desired result.
\end{proof}
Based on Lemma~\ref{lem:duality}, it suffices to construct a certificate $u$ with small $L^2$ norm. The unconditional cumulant construction underlying \cite[Theorem~2.2]{schramm2022computational}, also discussed in \cite[\S5.1]{sohn2025sharp}, uses the joint cumulants of $\btheta_1$ together with $\{ x_{ij}:(i,j) \in \cE_n \}$. In the present notation, its certificate representation is
\[
    u^\sharp(\btheta,\bW)=\sum_{|\alpha|\leq D} \frac{(1/\sqrt n)^{|\alpha|}}{\sqrt{\alpha!}} \cC_\alpha h_\alpha(\bW),
\]
which is a degree-$D$ certificate for $\btheta_1$ (see \cite[\S5.1]{sohn2025sharp} for details). Thus, we have the following bound on low-degree MMSE:
\begin{equation*}
    \MMSE_{\le D} \ge m_2-\bbE\left[ u^\sharp(\btheta,\bW)^2 \right] = m_2- \sum_{|\alpha| \le D} \frac{ 1 }{ n^{|\alpha|}\alpha! } \cC_\alpha^2,
\end{equation*}
recovering the bound in \eqref{eq:simple-bound-mmse}.

For the conditional-cumulant bound, define
\begin{equation}{\label{eq:def-u-sharp}}
    u(\btheta,\bW)=\sum_{|\alpha|\le D} \frac{A_\alpha(\btheta)}{n^{|\alpha|/2}\sqrt{\alpha!}} h_\alpha(\bW).
\end{equation}
We show that $u$ is a degree-$D$ certificate for $\btheta_1$, thereby providing an alternative proof of Lemma~\ref{lem:bound-mmse}.

\begin{proof}[Alternative proof of Lemma~\ref{lem:bound-mmse}]
    Each $A_\alpha(\btheta)$ is square-integrable by the moment assumptions and conditional Jensen's inequality. Since $\bW$ is independent of $\btheta$, the finite sum defining $u$ is therefore square-integrable. We use the conditional moment--cumulant identity \eqref{eq:conditional-moment-cumulant} from Lemma~\ref{lem:conditional-moment-cumulant} to verify the certificate property on a polynomial basis.

    We now verify \eqref{eq:certificate} for \eqref{eq:def-u-sharp}. Since $h_\alpha(\bW)=\prod_{e\in\cE_n}h_{\alpha_e}(\bW_e)$ and the coordinates of $\bW$ are independent standard Gaussian variables, the family $(h_\alpha(\bW))_\alpha$ is orthonormal. Applying Lemma~\ref{lem:Hermite-expansion} to $\bY=(1/\sqrt n) \btheta\btheta^{\top}+\bW$ gives
    \[
    \bbE[h_\alpha(\bW)h_\beta(\bY)\mid\btheta]=
    \begin{cases}
        \displaystyle
        \frac{\sqrt{\beta!}}{\sqrt{\alpha!}(\beta-\alpha)!} (1/\sqrt n)^{|\beta|-|\alpha|}x^{\beta-\alpha}, &\alpha\leq\beta,\\
        0,&\text{otherwise}.
    \end{cases}
    \]
    Fix $\beta$ with $|\beta|\leq D$. Since $\alpha\leq\beta$ implies $|\alpha|\leq D$, substituting the definition of $u$ into the last identity yields
    \begin{align*}
        \bbE[u(\btheta,\bW)h_\beta(\bY)] &=\frac{(1/\sqrt n)^{|\beta|}}{\sqrt{\beta!}} \sum_{\alpha\leq\beta}\binom{\beta}{\alpha} \bbE[A_\alpha(\btheta)x^{\beta-\alpha}]  \\
        &= \frac{(1/\sqrt n)^{|\beta|}}{\sqrt{\beta!}} \sum_{\alpha\leq\beta}\binom{\beta}{\alpha} \bbE[\cK_\alpha(\bR)x^{\beta-\alpha}]  \\
        &=\frac{(1/\sqrt n)^{|\beta|}}{\sqrt{\beta!}} \bbE[\btheta_1x^\beta],
    \end{align*}
    where the second equality follows from the definition of $A_\alpha$ and the measurability of $x^{\beta-\alpha}$ with respect to $\btheta$, and the last equality follows from first conditioning on $\bR$ and then taking expectations in \eqref{eq:conditional-moment-cumulant}. On the other hand, applying Lemma~\ref{lem:Hermite-expansion} to $\bY=(1/\sqrt n) \btheta\btheta^{\top}+\bW$ gives
    \[
       \bbE[\btheta_1h_\beta(\bY)] =\frac{(1/\sqrt n)^{|\beta|}}{\sqrt{\beta!}} \bbE[\btheta_1x^\beta].
    \]
    Consequently, $\bbE[uh_\beta(\bY)]=\bbE[\btheta_1h_\beta(\bY)]$ for every $|\beta|\leq D$. As the polynomials $\{h_\beta(\bY):|\beta|\leq D\}$ form a basis for the scalar polynomials of total degree at most $D$ in the upper-triangular observations, \eqref{eq:certificate} follows, so $u$ is a degree-$D$ certificate. Finally, orthonormality of the Hermite polynomials conditionally on $\btheta$ gives
    \[
        \bbE[u^2]=\sum_{|\alpha|\le D}
            \frac{1}{n^{|\alpha|}\alpha!}
            \bbE[A_\alpha(\btheta)^2].
    \]
    Applying Lemma~\ref{lem:duality} proves \eqref{eq:delicate-bound-mmse}.
\end{proof}

\subsection*{Acknowledgments}

This work is partially supported by the NSFC Major Program (Project No. 12595294), the NSFC Key Program (Project No. 12231002), and the New Cornerstone Investigator Program (Project No. NCI202501).

\subsection*{Statement on AI use}

Most of the formal arguments in this paper were initially developed using GPT-6. The primary human input is the high-level idea to apply Jensen's inequality with an appropriate sigma-field that contains the information in the auxiliary Gaussian channel $\bR$, which leads to the weaker bound in Lemma~\ref{lem:conditional-mmse-weak}. During this process, we interacted extensively with GPT (GPT-5.6 and GPT-6). GPT-6 then suggested obtaining the sharper bound using the more delicate sigma-field in \eqref{eq:final-choice-cG}, and identified the conditional cumulant bound stated in Lemma~\ref{lem:bound-mmse} and concluded the proof following successive rounds of interactive refinement. The author verified all mathematical content himself, wrote the introduction and the complete proof in their present form, and takes full responsibility for the content of the paper.

\bibliographystyle{alpha}
\bibliography{AMPLowDeg}

\end{document}